\documentclass{amsart}

\usepackage{amssymb,amsfonts}
\usepackage[all,arc]{xy}
\usepackage{enumerate}
\usepackage{mathrsfs}
\usepackage{verbatim}
\usepackage{microtype}
\usepackage{mathtools}

\newcounter{derpp}

\newtheorem{thm}{Theorem}[section]
\newtheorem*{thm*}{Theorem}
\newtheorem{cor}[thm]{Corollary}
\newtheorem{prop}[thm]{Proposition}
\newtheorem{lem}[thm]{Lemma}

\newtheorem*{sconj}{Serre's Conjecture}
\newtheorem*{conj*}{Conjecture I}

\newtheorem{mthm}[derpp]{Theorem}

\theoremstyle{definition}
\newtheorem{defn}[thm]{Definition}
\newtheorem{defns}[thm]{Definitions}

\newtheorem{exmp}[thm]{Example}

\newtheorem{notns}[thm]{Notations}

\theoremstyle{remark}

\DeclareMathOperator{\Spec}{Spec}
\DeclareMathOperator{\Proj}{Proj}

\DeclareMathOperator{\Tor}{Tor}
\DeclareMathOperator{\gr}{gr}
\DeclareMathOperator{\Supp}{Supp}
\DeclareMathOperator{\ord}{ord}

\DeclareMathOperator{\red}{red}
\DeclareMathOperator{\bH}{H}

\newcommand{\mcf}{\mathcal{F}}
\newcommand{\mcg}{\mathcal{G}}
\newcommand{\mco}{\mathcal{O}}

\newcommand{\fm}{\mathfrak{m}}
\newcommand{\fp}{\mathfrak{p}}
\newcommand{\fq}{\mathfrak{q}}
\newcommand{\fa}{\mathfrak{a}}
\newcommand{\fb}{\mathfrak{b}}
\newcommand{\fd}{\mathfrak{d}}
\newcommand{\ds}{\displaystyle}

\newcommand{\mcm}{\mathcal{M}}
\newcommand{\oth}{\widehat{\otimes}}

\newcommand{\opi}{\overline{\pi}}
\newcommand{\dopi}{\overline{\pi} \otimes 1 - 1 \otimes \overline{\pi}}
\newcommand{\GG}{\mathbf{G}}

\makeatletter
\let\c@equation\c@thm
\makeatother
\numberwithin{equation}{section}

\title{Intersection Multiplicity of Serre in the Unramified Case}
\author{C. Skalit}
\address{Department of Mathematics \\
University of Chicago \\
5734 S. University Avenue \\
Chicago, IL 60637}
\email[Chris Skalit]{cskalit@math.uchicago.edu}
\thanks{This work was funded in part by NSF grant DMS-1006610.}
\begin{document}
\begin{abstract}We describe here some recent progress pertaining to the Serre Intersection Multiplicity Conjecture. In particular, we show that if $A$ is unramified, then just as in the equicharacteristic case, the intersection multiplicity of two modules is bounded below by the product of their Hilbert-Samuel multiplicities. We also explain, in terms of the blowup of $\Spec A$, the geometric significance of achieving this lower bound.
\end{abstract}
\maketitle
\setcounter{derpp}{0}
\section*{Introduction}
Let $(A, \fm)$ be a regular local ring. For two finitely-generated $A$-modules $M$ and $N$ with $\ell(M \otimes_A N) < \infty$, Serre \cite[V]{Serre} defines the ``intersection multiplicity'' via
\[ \chi^A(M,N) = \sum_{i=0}^{\dim A}{(-1)^i\ell(\Tor_i^A(M,N))}. \]
The formula, which was originally proposed to define the intersection product of properly-meeting cycles on an algebraic variety, has garnered interest in its own right, thanks to the following conjecture:
\begin{sconj}\cite[V]{Serre} Let $(A, \fm)$ be a regular local ring. Suppose that $M$ and $N$ are finitely-generated modules with $\ell(M \otimes_A N) < \infty$. Then the following statements hold:
\begin{itemize}
\item[(a)] $\chi^A(M,N) \geq 0$.
\item[(b)] $\dim M + \dim N \leq \dim A$.
\item[(c)] $\chi^A(M,N) > 0$ if and only if $\dim M + \dim N = \dim A$.
\end{itemize}
\end{sconj}

Serre [\emph{loc. cit.}] showed (b) holds in general and that (a) and (c) are also true provided that $A$ is either equicharacteristic or of mixed-characteristic and unramified. While the positivity in (c) remains open, Gabber has shown, using de Jong's theory of regular alterations \cite{deJong}, that (a) holds in general. While Gabber never published his result, accounts may be found in \cite{Berthelot} and \cite{Hochster}. The ``only-if'' in (c) --- that is, the vanishing of $\chi^A(M,N)$ when $\dim M + \dim N < \dim A$ --- was proved using K-Theoretic techniques by Gillet and Soul\'{e} \cite{Gillet}. An alternate proof, due to P. Roberts, may be found in \cite{Roberts}. 

For the case of an equicharacteristic $A$, Serre's proof of (c) shows that if $M$ and $N$ are of complimentary dimension (that is, when $\dim M + \dim N = \dim A$), then, in fact, one has $\chi^A(M,N) \geq e(M)e(N)$ where we denote by $e(M)$ the Hilbert-Samuel multiplicity of $M$ (with respect to $\fm$). Our first result is an extension of this lower bound to the unramified case:

\begin{mthm} Suppose that $(R_0,\pi R_0)$ is a discrete valuation ring with perfect residue field. Let $(A,\fm)$ be a regular local ring containing $R_0$ such that $\pi \in \fm - \fm^2$. If $M$ and $N$ are finitely-generated $A$-modules with $\dim M + \dim N = \dim A$ and $\ell(M \otimes_A N) < \infty$, then $\chi^A(M,N) \geq e(M)e(N)$.\footnote{Although it has never before appeared in print, this inequality was previously known to O. Gabber.}
\end{mthm}

Note that for $R_0 = \mathbb{Z}_{p\mathbb{Z}}$, we have precisely the case where $A$ is unramified in the traditional sense.

To prove Theorem A, we may reduce to when $A$ is complete and hence a power-series ring over $R$, a complete DVR. We may further suppose that at least one of $M$ or $N$ is $R$-flat. In this case, $\chi^A(M,N) = e_\fd(M \oth_R N)$ where $\fd \subseteq A \oth_R A$ is the ``diagonal ideal.'' What we then actually prove is the stronger result that $e(M \oth_R N) \geq e(M)e(N)$. This and other properties regarding the completed tensor product over a complete DVR are examined in detail in Section 2; the proof of Theorem A appears in Section 3.1.

Having established a lower bound for $\chi^A(M,N)$, we now attempt to understand under what circumstances it is achieved. Let $(A,\fm)$ be regular and put $X = \Spec A$. Consider the closed subschemes $Y = \Spec(A/\fp)$ and $Z = \Spec(A/\fq)$ where $\fp$ and $\fq$ are prime ideals. Assume that $Y \cap Z$ is a point and that $\dim Y + \dim Z = \dim X$. Geometric intuition suggests that if $\chi^A(A/\fp,A/\fq) \neq e(A/\fp)e(A/\fq)$, then $Y$ and $Z$ do not meet transversely. In other words, their tangent cones intersect nontrivially in the sense that  $\dim(\gr(A/\fp) \otimes_{\gr A} \gr(A/\fq)) > 0$. This implication was first proved rigorously by B. Tennison \cite{Tennison}.

We can obtain a more quantitative version of Tennison's result by passing to $\widetilde{X}$, the blowup of $X$ along the ideal $\fm$. If  $\mcf$ and $\mcg$ are coherent $\mco_{\widetilde{X}}$-modules with supports meeting in the exceptional divisor $E \subseteq \widetilde{X}$, we can define a ``global'' intersection multiplicity via
\[ \chi^{\mco_{\widetilde{X}}} (\mcf, \mcg) := \sum_{i,j \geq 0}{ (-1)^{i+j} \ell (\bH^i(\widetilde{X},(\Tor_j^{\mco_{\widetilde{X}}}(\mcf,\mcg))))} \]
where we denote by $\bH^i$ the (Zariski) sheaf-cohomology groups and $\ell$ the length of each as an $A$-module. Writing $\widetilde{Y}$ and $\widetilde{Z}$ for the strict transforms of $Y$ and $Z$, we have a formula \cite[Example 20.4.3]{Fulton} connecting the two multiplicities:
\[ \chi^A(A/\fp,A/\fq) = e(A/\fp)e(A/\fq) + \chi^{\mco_{\widetilde{X}}}(\mco_{\widetilde{Y}},\mco_{\widetilde{Z}}). \]
We note that Tennison's theorem now follows from the fact that $\widetilde{Y} \cap \widetilde{Z} = \varnothing$ precisely when the tangent cones of $Y$ and $Z$ intersect trivially (see Lemma \ref{st_explained}).

S. Dutta has investigated the natural question of whether $\chi^{\mco_{\widetilde{X}}}(\mco_{\widetilde{Y}},\mco_{\widetilde{Z}}) \geq 0$ in \cite{Dutta_Blowup}. When $A$ is unramified in the sense of Theorem A, this inequality is guaranteed. We now propose a stronger positivity conjecture which asserts that $\chi^{\mco_{\widetilde{X}}}(\mco_{\widetilde{Y}},\mco_{\widetilde{Z}}) > 0$ whenever $\widetilde{Y} \cap \widetilde{Z}$ is nonempty. Algebraically, this amounts to saying:

\begin{conj*}Let $A$ be a regular local ring and suppose that $M$ and $N$ are equidimensional, finitely-generated modules such that $\dim M + \dim N = \dim A$ and $\ell(M \otimes_A N) < \infty$. Then $\chi^A(M,N) \geq e(M)e(N)$ with equality occurring if and only if $\dim (\gr M \otimes_{\gr A} \gr N) = 0$.
\end{conj*}

It should go without saying that Conjecture I implies the positivity in part (c) of Serre's Conjecture. We shall therefore only concern ourselves with regular local rings $A$ for which positivity is already known. More precisely, we aim to show that if $\chi^A(M,N) = e(M)e(N)$ then $\dim(\gr M \otimes_{\gr A} \gr N) = 0$; the reverse implication, of course, follows from the work of Tennison \cite{Tennison}. We also remark that since only the top-dimensional components of $\Supp M$ and $\Supp N$ contribute to $\chi^A(M,N)$, we need to impose equidimensional hypotheses to exclude the obvious counterexamples (see Example \ref{exmp_equidim}).

In Section 3.3, we prove:
\begin{mthm}\label{i_equichar}If $A$ is equicharacteristic, then Conjecture I is true.
\end{mthm}
When $A$ is essentially smooth over a field, Theorem \ref{i_equichar} may be proved via the techniques of Fulton-MacPherson intersection theory. Our method, which is purely algebraic, relies instead on the celebrated Theorem of Rees \cite{Rees} that relates multiplicities to integral closures of ideals. The second key ingredient in our approach is the following isomorphism of P. Samuel \cite{Samuel}, which states that for $K$ a field and $A = K[[X_1, \cdots, X_n]]$, we have an isomorphism
\[ \gr M \otimes_K \gr N \stackrel{\sim}{\longrightarrow} \gr(M \oth_K N). \]
No such isomorphism exists in mixed-characteristic; in Section 2.4, we explain how we may, in certain cases, circumvent this difficulty.

For the remainder of this section, let us fix a discrete valuation ring $(R_0, \pi R_0)$ with perfect residue field. Let $(A,\fm)$ be any regular local ring containing $R_0$ such that $\pi \in \fm - \fm^2$.

\begin{mthm}With $A$ as above, Conjecture I is true for $A$-modules $M$ and $N$ under the additional assumption that $\Supp M \subseteq \Supp(A/\pi A)$.
\end{mthm}

In this case, we can in essence reduce ourselves to working with the equicharacteristic ring $A/\pi A$ where the result is already known by Theorem B. Details appear in Section 3.4.

We now turn to the case where the $A$-modules $M$ and $N$ are both $R_0$-flat. We first make the following crucial observation:
\begin{mthm}With $A$ as above, let $M$ and $N$ be $R_0$-flat $A$-modules such that $\dim M + \dim N = \dim A$ and $\ell(M \otimes_A N) < \infty$. If $\chi^A(M,N)= e(M)e(N)$, then $e(M) = e(M/\pi M)$ or $e(N) = e(N/\pi N)$.
\end{mthm}

With this result, we therefore see that when investigating Conjecture I for $R_0$-flat modules $M$ and $N$, we may always assume that at least one of the modules --- say $M$ --- satisfies $e(M) = e(M/\pi M)$. We now formulate:

\begin{mthm}Let $A$ be as above. Assume that $M = A/\fa$ and $N = A/\fb$ are equidimensional, $R_0$-flat quotients of $A$ and that $e(M) = e(M/\pi M)$. Then Conjecture I is true in the presence of any one of the following additional conditions:
	\begin{itemize}
		\item[(i)] $\opi$ is not contained in a height-one associated prime of $\gr M$;
		\item[(ii)] $e(N) = e(N / \pi N)$;
		\item[(iii)] $\dim M = \dim A - 1$;
		\item[(iv)] $\dim M = 1$;
		\item[(v)] $\ds \dim(\gr M \otimes_{\gr A} \gr N \otimes_{\gr A} \frac{\gr A}{\opi \gr A}) = 0$.
	\end{itemize}
\end{mthm}

The proof of Theorem E ties together two main ideas: one algebraic, the other geometric. In the first three cases, there exists a close-enough analogue of Samuel's isomorphism to directly apply the techniques used in Theorem B. In particular, we show (see Proposition \ref{gr_R}) that when $A$ is complete, there is a canonical surjection
\[ \frac{\gr (A/\fa) \otimes_k \gr (A/\fb) }{\dopi} \twoheadrightarrow \gr(A/\fa \oth_R A/\fb) \]
that induces a homeomorphism of the underlying Zariski spaces. For the latter two cases, we examine the blowups $\widetilde{Y}$ and $\widetilde{Z}$ of $\Spec(A/\fa)$ and $\Spec(A/\fb)$. Under these circumstances, the intersection of $\widetilde{Y}$ and $\widetilde{Z}$ is simple enough to allow for the direct computation of $\chi^{\mco_{\widetilde{X}}}(\mco_{\widetilde{Y}},\mco_{\widetilde{Z}})$; and we show, quite explicitly, that $\chi^{\mco_{\widetilde{X}}}(\mco_{\widetilde{Y}},\mco_{\widetilde{Z}}) = 0$ if and only if $\widetilde{Y} \cap \widetilde{Z} = \varnothing$. We treat the intersection multiplicity on $\widetilde{X}$ in Section 3.2; Theorems C, D, and E are proved in Section 3.4.

\section*{Acknowledgments}
The author wishes to thank his advisor, Madhav Nori, for many fruitful discussions and ideas. Additional thanks are owed to J. Peter May, whose suggestions helped shape the original draft into a more coherent narrative than it would have been otherwise. Finally, the author extends his thanks to the anonymous referee for carefully reading the manuscript and offering a number of insightful comments.

\section{Some Remarks on Hilbert-Samuel Multiplicity}
\setcounter{derpp}{0}
In this section, we fix notation and recall some elementary facts about Hilbert-Samuel multiplicity for later use. We give proofs when no suitable literature reference may be found.
\subsection{Hilbert Functions}
Let $S$ be an $\mathbb{N}$-graded ring, finitely generated by $S_1$ as an $S_0$-algebra where $S_0$ is Artin-local. For any finitely-generated, graded $S$-module $E$ of Krull dimension $d$, we can define the Hilbert function, $\ds f_E(n) := \ds \sum_{i=0}^n{\ell(E_i)}$. For $n >> 0$, one has that $f_E(n) = P_E(n)$ where $P_E \in \mathbb{Q}[x]$ is a degree-$d$ polynomial. The multiplicity $e(E)$ is defined via
\[ e(E) = \lim_{n \to \infty} \frac{d! \cdot f_E(n)}{n^d} = d! \cdot (\mbox{leading coefficient of $P_E$}).\]
So-called discrete derivatives are defined via $\Delta f_E(n) = f_E(n) - f_E(n-1) = \ell(E_{n})$. It is clear that for $n >> 0$, one has $\Delta f_E(n) = \Delta P_E(n)$ as well as 
\[ \Delta^iP_E(n) = \left\{ \begin{array}{ll} e(E) & i = d \\ 0 & i > d \end{array}\right. .\]
Let $(A,\fm)$ be a Noetherian local ring, $M$ a finitely-generated, $d$-dimensional $A$-module, and $\fa \subseteq A$ an ideal such that $\ell(M/\fa M) < \infty$. We denote by $f_{\fa,M}(n)$ the Hilbert function for $\ds \gr_\fa(M) := \bigoplus^{\infty}_{n=0}{\frac{\fa^n M}{\fa^{n+1}M}}$ and write $P_{\fa,M}(n)$ for the corresponding polynomial. We define the $\fa$-multiplicity via $e_\fa(M) = e(\gr_\fa(M))$. We will often write $e(M)$ and $\gr(M)$ when $\fa = \fm$. 

\subsection{Grothendieck Groups}
Now let $Y \subseteq \Spec A$ be a closed subset of dimension $d$. We shall write $\mcm(Y)$ for the abelian category of finitely-generated $A$-modules $M$ with support contained in $Y$ and will denote by $\GG_0(Y)$ the Grothendieck group on $\mcm(Y)$. Let $\fa \subseteq A$ be any ideal for which $V(\fa) \cap Y = \left\{\fm \right\}$; this is enough to guarantee that $\ell(M/\fa M) < \infty$ for all $M \in \mcm(Y)$. It is well known \cite[II-Prop.10]{Serre} that the mapping $e_\fa(d,-)$, which sends $M \in \mcm(Y)$ to $\Delta^dP_{\fa,M} \in \mathbb{Z}$, is additive over short-exact sequences in $\mcm(Y)$ and so defines a homomorphism $\GG_0(Y) \to \mathbb{Z}$.

When computing Hilbert-Samuel multiplicities or intersection multiplicities, it is often convenient to pass to the primes associated to the modules in question. To that end, we have the following lemma:
\begin{lem}\label{decomp}Let $A$ and $Y$ be as above. Then for each $M \in \mcm(Y)$, there is an $M'' \in \mcm(Y)$ with $\dim M'' < d$ and an equation in $\GG_0(Y)$ of the form
\[ [M] = \sum_{\mathclap{\substack{\fp \in Y \\ \dim A/\fp = d}}}{\hspace{1mm}\ell(M_\fp) [A/\fp]} + [M'']. \]
\end{lem}
\begin{proof}We proceed by induction on $\ds a(M) = \sum_{\mathclap{\dim A/\fp = d}}{\hspace{1mm}\ell(M_\fp)}$. If $a(M) = 0$, we have $\dim M < d$, and the proof is complete.  If $a(M) > 0$, we choose some $\fp \in \Supp M$ such that $\dim A/\fp = d$. Since $\dim M = d$, this $\fp$ is minimal and hence associated, thereby giving rise to an exact sequence
\[0 \to A/\fp \to M \to N \to 0.\]
For every $d$-dimensional component $\fq$ of $Y$, we see from the above sequence that $\ell(N_\fq) = \ell(M_\fq)$ for $\fq \neq \fp$ while $\ell(N_\fp) = \ell(M_\fp) - 1$. Hence, $a(N) = a(M) - 1$ and the proof is complete by induction.
\end{proof}

Invoking Lemma \ref{decomp} to $M$ in $\GG_0(\Supp M)$ and applying $e_\fa(d,-)$ yields the familiar ``additivity formula'' \cite[V-2]{Serre}.
\begin{cor}\label{add}(Additivity Formula) Let $M$ be a $d$-dimensional module over the local ring $A$ and let $\fa$ be an ideal for which $\ell(M/\fa M) < \infty$. Then
\[e_\fa(M) = \sum_{\mathclap{\substack{\fp \in \Supp M \\ \dim A/\fp = d}}}{\,\ell(M_\fp)e_{\fa}(A/\fp)}.\]
\end{cor}

\subsection{Multiplicity Modulo a Divisor}
\begin{prop}\label{ss_mult}\cite[Thm. 3.2(1)]{Flenner} Let $(A,\fm)$ be a Noetherian local ring with $\fa \subseteq \fm$. Let $M$ be any finitely generated module such that $M/\fa M$ is Artinian. Suppose that $x \in \fa^t$ is a non-zerodivisor on $M$. Denote by $\overline{x}$ the image of $x$ in $\fa^t/\fa^{t+1}$. Then $e_\fa(M/xM) \geq t \cdot e_\fa(M)$ with equality occurring if and only if $\gr_\fa(M)/\overline{x} \gr_\fa(M)$ has lesser dimension than $\gr_\fa(M)$.
\end{prop}

We also include the following companion statement which shows that comparing the multiplicities of $M$ and $M/xM$ really amounts to making the same comparison for $A/\fp$ where $\fp$ is a minimal, associated prime of $M$.

\begin{prop}\label{decomp2}Let $(A,\fm)$ be a Noetherian local ring with $\fa \subseteq \fm$. Let $M$ be a finitely-generated, $d$-dimensional $A$-module such that $M/\fa M$ is Artinian. Let $x \in \fa^t$ be a non-zerodivisor for both $A$ and $M$. Then $e_\fa(M/xM) = t \cdot e_\fa(M)$ if and only if $e_\fa((A/\fp)/x(A/\fp)) = t \cdot e_\fa(A/\fp)$ for each $\fp \in \Supp(M)$ such that $\dim(A/\fp) = d$.
\end{prop}
\begin{proof}Consider the map $\Phi$, given by intersecting with $x$:
\[\begin{array}{rcl}
\GG_0(\Supp M) & \stackrel{\Phi}{\longrightarrow} & \GG_0(\Supp M \cap \Supp (A/xA)) \\
\left[N\right]          & \longmapsto   & [N \otimes_A A/xA] - [\Tor_1^A(N, A/xA)].
\end{array} \]
From Lemma \ref{decomp}, we know that in $\GG_0(\Supp M)$, we may write:
\[ \ds [M] = \sum_{\mathclap{\substack{\fp \in \Supp M \\ \dim A/\fp = d}}}{\hspace{1mm}\ell(M_\fp) [A/\fp]} + [M''] \]
where $\dim(M'') < d$. For those $\fp$ subject to $\dim(A/\fp) = d$, we have $x \notin \fp$ as $\fp$ is associated. By applying $\Phi$, we therefore obtain, in $\GG_0(\Supp M \cap \Supp (A/xA))$, the relation

\[ [M/xM] = \sum_{\mathclap{\substack{\fp \in \Supp M \\ \dim A/\fp = d}}}{\hspace{1mm}\ell(M_\fp) [A/\fp \otimes_A A/xA]} + \Phi([M'']). \]

We can read off the multiplicity $e_\fa(M/xM)$ by simply applying $e_\fa(d-1,-)$. Thus, our statement will be proved once we show that the virtual module $\Phi([M''])$ is represented in $\GG_0(\Supp M \cap \Supp A/xA)$ as a sum of modules having dimension less than $d-1$. For this, we observe that repeated applications of Lemma \ref{decomp} will show that in $\GG_0(\Supp M)$, $[M'']$ may be written as a sum of $[A/\fq]$'s, where $\dim A/\fq \leq d-1$. If $x \in \fq$ then $\Tor_1^A(A/\fq,A/xA) \cong A/\fq = A/\fq \otimes_A A/xA$, so $\Phi([A/\fq]) = 0$ in this case. Otherwise, we have that $x$ is $A/\fq$-regular, meaning that $\Phi([A/\fq]) = [A/\fq \otimes_A A/xA]$ and thus is represented by a module of dimension strictly less than $d-1$.
\end{proof}
\subsection{Integral Closure of Ideals}
We state here some elementary facts about the integral closure of ideals that will be used in the sequel. We refer the reader to \cite{Huneke} for proofs.

\begin{defn}Let $\fa$ be an ideal in a ring, $A$. An element $b \in A$ is said to be \emph{integral over $\fa$} if it satisfies an equation of the form
\[ b^n + a_1b^{n-1} + \cdots + a_{n-1}b + a_n = 0 \]
where each $a_i$ belongs to $\fa^i$. The \emph{integral closure of $\fa$}, denoted $\overline{\fa}$, is the collection of all $b \in A$ that are integral over $\fa$. Note that $\overline{\fa}$ is an ideal by \cite[1.3.1]{Huneke}.   
\end{defn}

The following lemma is an immediate consequence of the definitions.
\begin{lem}\label{persistence}Let $A \to B$ be a ring morphism and suppose that $\fa \subseteq \fb$ are ideals of $A$. Then if $\fb \subseteq \overline{\fa}$, it follows that $\fb B \subseteq \overline{\fa B}$. 
\end{lem}

We shall be interested in the connection between integral closure and Hilbert-Samuel multiplicity. The first result in this direction is a fairly elementary observation:
\begin{prop}\label{rees_easy}\cite[11.2.1]{Huneke} Let $(A,\fm)$ be a Noetherian local ring and suppose that $\fa \subseteq \fb$ are ideals such that $\fb \subseteq \overline{\fa}$. If $M$ is a finitely generated module with $\ell(M/\fa M) < \infty$, then $e_\fa(M) = e_\fb(M)$. 
\end{prop}

The converse, that multiplicity alone determines integral closure, requires additional hypotheses and is considerably more subtle. It is the Theorem of Rees alluded to in the introduction.

\begin{thm}\label{rees_thm}\cite{Rees} Let $(A,\fm)$ be a formally-equidimensional Noetherian local ring and suppose that $\fa \subseteq \fb$ are two $\fm$-primary ideals. Then $e_\fa(A) = e_\fb(A)$ if and only if $\fb$ is contained in the integral closure of $\fa$.
\end{thm}
 
\section{Multiplicities of Completed Tensor Products}
For this entire section, we fix a complete DVR $(R,\pi R)$ with residue field $k$. If $A = R[[X_1, \cdots, X_m]]$, computing $\chi^A(M,N)$ for two $R$-flat $A$-modules amounts to computing the multiplicity of $M \oth_R N$ with respect to a certain ``diagonal ideal'' as we shall recall in the next section. For proving the lower bound in this case, it certainly suffices to show that $e(M \oth_R N) \geq e(M)e(N)$. To do so, we break the symmetry of the situation by introducing a second power-series ring $B = R[[Y_1, \cdots, Y_n]]$, considering a $B$-module $N$, and examining the multiplicity of $M \oth_R N$, now regarded as an $A \oth_R B$-module. For generalities concerning the completed tensor product, we refer the reader to \cite[V-6]{Serre}.

\subsection{The Lower Bound}
\begin{lem}\label{ct_dim}Let $A = R[[X_1, \cdots, X_n]]$ and let $B=R[[Y_1, \cdots, Y_m]]$. If $M$ and $N$ are finitely-generated $A$ and $B$-modules respectively, then the following statements hold:
\begin{itemize}
\item[(a)] $\dim (M \oth_R N) \leq \dim M + \dim N$.
\item[(b)] If $N$ is flat over $R$, then $\dim(M \oth_R N) = \dim M + \dim N - 1$.
\end{itemize} 
\end{lem}
\begin{proof}
Let $\fp_1, \cdots, \fp_k$ be the minimal primes of $\Supp M$ and $\fq_1, \cdots, \fq_\ell$ the minimal primes of $\Supp N$. Put $C = A \oth_R B$ and let $\phi_A$ and $\phi_B$ be the canonical maps from $\Spec C$ to $\Spec A$ and to $\Spec B$. We may write $M \oth_R N = (M \otimes_A C) \otimes_C (N \otimes_B C)$. Note that $\Supp(M \otimes_A C) = \phi_A^{-1}(\Supp(M))$ and similarly for $N$. We therefore obtain the following equalities:
\[
\begin{array}{rcl}
\ds \Supp(M \oth_R N) & = & \ds \Supp(M \otimes_A C) \cap \Supp(N \otimes_B C) \vspace{2mm}\\
                      & = & \ds \phi_A^{-1}(\Supp M) \cap \phi_B^{-1}(\Supp N) \vspace{2mm}\\
											& = & \ds \bigcup_{i,j}{\phi_A^{-1}(\Supp A/\fp_i) \cap \phi_B^{-1}(\Supp(B/\fq_j))} \\
											& = & \ds \bigcup_{i,j}{\Supp(A/\fp_i \otimes_A C) \cap \Supp(B/\fq_j \otimes_B C)} \\
											& = & \ds \bigcup_{i,j}{\Supp(A/\fp_i \oth_R B/\fq_j)}
\end{array}
\]
One always has an isomorphism
\[\frac{A/\fp_i \oth_R B/\fq_j}{\pi(A/\fp_i \oth_R B/\fq_j)} \cong \frac{A/\fp_i}{\pi(A/\fp_i)} \oth_k \frac{B/\fq_j}{\pi(B/\fq_j)}, \]
and from it, (a) readily follows.

For (b), we note that if $N$ is $R$-flat, $\pi$ must act as a non-zerodivisor on each $B/\fq_j \hookrightarrow N$. In this case, $\pi$ either annihilates $A/\fp_i \oth_R B/\fq_j$ or is a non-zerodivisor, depending on whether $\pi \in \fp_i$. In either case, we have
\[ \dim(A/\fp_i \oth_R B/\fq_j) = \dim(A/\fp_i) + \dim(B/\fq_j)- 1,\]
and the conclusion follows.
\end{proof}

\begin{thm}\label{main_ineq}Let $A = R[[X_1, \cdots, X_n]]$ and let $B=R[[Y_1, \cdots, Y_m]]$. Let $M$ and $N$ be finitely-generated $A$ and $B$-modules respectively. If $N$ is flat over $R$, then $e(M \oth_R N) \geq e(M)e(N)$.
\end{thm}
\begin{proof}It is well-known \cite[Thm.29.1]{Matsumura} that there is a faithfully-flat base extension $R \to S$ where $S$ is a complete DVR with uniformizer $\pi$ and $S/\pi S$ is algebraically closed. Put $A' = S \oth_R A = S[[X_1, \cdots, X_m]]$ and similarly define $B'$, $M'$, and $N'$. We note that $A \to A'$ is a faithfully-flat morphism for which $\fm_A A' = \fm_{A'}$ and $A'/\fm_{A'}$ is the algebraically-closed field, $S/\pi S$. Since all relevant quantities will remain unchanged if we replace $A$, $B$, $M$, and $N$ with $A'$, $B'$, $M'$ and $N'$, we shall henceforth assume that $k =R/ \pi R$ is algebraically closed.

We next observe that since $N$ is $R$-flat, we have an induced homomorphism of Grothendieck groups:
\[ \begin{array}{rcl}
\GG_0(\Supp M) & \longrightarrow & \GG_0(\Supp (M \oth_R N)) \\
										  \left[H\right]      & \longmapsto     & [H \oth_R N]  
\end{array} \]
By Lemma \ref{decomp}, we may express $[M]$ in $\GG_0(\Supp M)$ via
\[ \ds [M] = \sum_{\mathclap{\substack{\fq \in \Supp M \\ \dim A/\fq = d}}}{\hspace{1mm}\ell(M_\fq) [A/\fq]} + [M''] \]
where $\dim M = d$, the $\fq$ are prime, and $\dim M'' < d$. One therefore obtains in $\GG_0(\Supp (M \oth_R N))$ the relation
\[ [M \oth_R N] = \sum_{\mathclap{\substack{\fq \in \Supp M \\ \dim A/\fq = d}}}{\hspace{1mm}\ell(M_\fq) [A/\fq \oth_R N]} + [M'' \oth_R N]. \]
Lemma \ref{ct_dim} shows $\dim (M'' \oth_R N) = \dim M'' + \dim N - 1 < \dim M + \dim N - 1$. Applying $e(\dim M + \dim N - 1, -)$ to the above expression therefore yields
\[ e(M \oth_R N) = \sum_{\mathclap{\substack{\fq \in \Supp M \\ \dim A/\fq = d}}}{\hspace{1mm}\ell(M_\fq) e(A/\fq \oth_R N)}, \]
so by the additivity formula (\ref{add}) it suffices to prove the claim for $M = A/\fp$ where $\fp$ is a prime ideal.

\textbf{Case 1:} (\emph{$M = A/fA$ for some nonzero $f \in A$.})
Put $C = A \oth_R B$. In this case, $M \oth_R N \cong (A \oth_R N)/f(A \oth_R N)$ as $C$-modules. Denote by $\fm_A$ the maximal ideal of $A$ and $a = \ord(f)$ the highest power of $\fm_A$ to which $f$ belongs. Consequently, we have $f \in \fm_{C}^a$ whence, by Proposition \ref{ss_mult}, we obtain $e(M \oth_R N) \geq a e(A \oth_R N)$. However, since $A$ is regular, $e(M) = e(A/fA) = a$ \cite[11.2.8]{Huneke}, and since $A \oth_R N \cong N[[X_1, \cdots, X_m]]$, we have $e(N) = e(A \oth_R N)$, thereby proving the desired inequality in this case. 

\textbf{Case 2:} (\emph{$M = A/\fp$ for some prime ideal, $\fp$.})
Since $A/\fm_A$ is infinite, we know (see, for example, \cite[Thm. 14.14]{Matsumura}) that there exists a system of parameters $x_1, \cdots, x_d \in A/\fp$ such that the integral closure of the ideal they generate is equal to all of $\fm_{A/\fp}$. Let $R[[T_1, \cdots, T_d]] \to A/\fp$ be the morphism mapping $T_i$ to $x_i$. This morphism is finite. From dimensional considerations, its kernel is prime and generated by a single non-zero element $g$. Putting $A_0 = R[[T_1, \cdots, T_d]]/(g)$, we have an induced finite, integral extension of rings $A_0 \hookrightarrow A/\fp$. Let $r = [K(A/\fp):K(A_0)]$ where for a domain $D$ we denote by $K(D)$ its fraction field. We then obtain a short exact sequence of $A_0$-modules
\[ 0 \to (A_0)^r \to A/\fp \to E \to 0 \]
for which $E \otimes_{A_0} K(A_0) = 0$. Thus, $\dim E < d = \dim A_0$, so by applying $e_{\fm_{A_0}}(d,-)$, we see that $e_{\fm_{A_0}}(A/\fp) = r e(A_0)$. We make two observations at this point: First, since $k$ is algebraically closed, $A/\fp$ and $A_0$ have the same residue field. Thus, given an Artinian $A/\fp$ module, it does not matter whether we compute its length as an $A/\fp$ or as an $A_0$ module. In particular, when computing the multiplicity $e_{\fm_{A_0}}(A/\fp)$ of $A/\fp$ as an $A_0$-module, it is the same as computing $e_{\fm_{A_0}(A/\fp)}(A/\fp)$, the multiplicity of $A/\fp$ when regarded as a module over itself. Second, we note that since $(x_1, \cdots, x_d) \subseteq \fm_{A_0}(A/\fp)$, we have $\fm_{A/\fp}$ is contained in the integral closure of the extended ideal, $\fm_{A_0}(A/\fp)$. By Proposition \ref{rees_easy}, we are therefore assured that $e(M) = e(A/\fp) = e_{\fm_{A_0}}(A/\fp) = r e(A_0)$.

Since $N$ is $R$-flat, we can apply $-\oth_R N$ to obtain the following exact sequence of $A_0 \oth_R B$-modules:
\[ 0 \to (A_0 \oth_R N)^r \to (M \oth_R N) \to (E \oth_R N) \to 0 \]
If we consider the morphism of rings $A_0 \oth_R B \to A/\fp \oth_R B$, it is immediate from Lemma \ref{persistence} that the integral closure of $\fm_{A_0 \oth_R B}$ in $A/\fp \oth_R B$ contains all of $\fm_{A/\fp \oth_R B}$, thus ensuring that $e_{\fm_{A_0 \oth_R B}}(M \oth_R N) = e(M \oth_R N)$. We also note that by Lemma \ref{ct_dim}, $\dim(E \oth_R N) < \dim(M \oth_R N) = \dim(A_0 \oth_R N)$. With these two considerations, we can apply $e_{\fm_{A_0 \oth_R B}}(\dim(A_0 \oth_R N), -)$ to obtain $e(M \oth_R N) = r e(A_0 \oth_R N)$. Now, $A_0$ is a power-series ring modulo a principal ideal, so by Case 1, we conclude that
\[ e(M \oth_R N) = r e(A_0 \oth_R N) \geq r e(A_0)e(N) = e(M)e(N). \]
\end{proof}

\subsection{When Equality Holds} It's not difficult to show that the inequality in Theorem \ref{main_ineq} can be strict, even in the case of an equicharacteristic DVR:
\begin{exmp} Let $R = \mathbb{C}[[T]]$ with $M = N  = R[[X]]/(T-X^2)$, then $e(M)$ and $e(N)$ are both $1$, but $M \oth_R N \cong \mathbb{C}[[X,Y]]/(X^2 - Y^2)$ and so has multiplicity $2$.
\end{exmp} In this example, the inequality $e(M \oth_R N) > e(M)e(N)$ may be attributed to the fact that in both $M$ and $N$, the uniformizer $T$ is identified with an element in the square of the maximal ideal. The following theorem says that this is essentially the only reason why the inequality can be strict.

\begin{thm}\label{ctp_eq}Let $A = R[[X_1, \cdots, X_m]]$ and $B = R[[Y_1, \cdots, Y_n]]$. Let $M$ and $N$ be finitely-generated $A$ and $B$ modules respectively, both of which are flat over $R$. Then $e(M \oth_R N) = e(M)e(N)$ if and only if $e(M) = e(M/\pi M)$ or $e(N) = e(N/\pi N)$.
\end{thm}
\begin{proof}Just as in the proof of Proposition, \ref{main_ineq}, we may assume that $k = R/\pi R$ is algebraically closed. By Lemma \ref{decomp}, we have, in $\GG_0(\Supp M)$ and $\GG_0(\Supp N)$ respectively, the equalities
\[ \ds [M] = \sum_{\mathclap{\substack{\fp \in \Supp M \\ \dim A/\fp = \dim M}}}{\hspace{1mm}\ell(M_\fp) [A/\fp]} + [M''] \mbox{\hspace{5mm}and\hspace{5mm}}
  \ds [N] = \sum_{\mathclap{\substack{\fq \in \Supp N \\ \dim B/\fq = \dim N}}}{\hspace{1mm}\ell(N_\fq) [B/\fq]} + [N''] \]
where $\dim M'' < \dim M$ and $\dim N'' < \dim N$.
Since $M$ is $R$-torsion free (i.e. flat over the DVR $R$), so too is each $A/\fp \hookrightarrow M$. The same can be said for any $B/\fq \hookrightarrow N$. We can ``multiply'' $[M]$ and $[N]$ together via the map
\[\begin{array}{rcl}
\GG_0(\Supp M) \otimes \GG_0(\Supp N) & \longrightarrow & \GG_0(\Supp M \oth_R N) \\
 \left[H\right] \otimes \left[H'\right]                 & \longmapsto & [H \oth_R H'] - [\widehat{\Tor}_1^R(H,H')].
\end{array} \]
Afterward we obtain, in $\GG_0(\Supp M \oth_R N)$, the relation
\[ [M \oth_R N] = [M] \cdot [N] = \sum_{\mathclap{\substack{\fp \in \Supp M \\ \fq \in \Supp N \\ \dim A/\fp = \dim M \\ \dim B/\fq = \dim N}}}{\hspace{1mm}\ell(M_\fp)\ell(N_\fq)[A/\fp \oth_R B/\fq]} + \Gamma \]
where, by Lemma \ref{ct_dim}, $\Gamma$ is a linear combination of modules having lower dimension than $M \oth_R N$. From the additivity formula (\ref{add}), we see that 
\[e(M \oth_R N) = e(M)e(N) \mbox{ if and only if } e(A/\fp \oth_R B/\fq) = e(A/\fq)e(B/\fq) \]
for all $\fp$ and $\fq$ such that $A/\fp$ and $B/\fq$ have dimensions $\dim M$ and $\dim N$ respectively. By Proposition \ref{decomp2}, we also know that 
\[e(M) = e(M/\pi M) \mbox{ if and only only if } e(A/\fp) = e((A/\fp)/\pi(A/\fp))\]
for all $\fp \in \Supp M$ with $\dim(A/\fp) = \dim M$ (and similarly for $N$). From these two remarks, we are reduced to proving the theorem in the case that $M = A/\fp$ and $N = B/\fq$ where $\fp$ and $\fq$ are prime ideals. We shall henceforth assume that $M$ and $N$ are of this form.

We first tackle the ``if'' part and assume that $e(M) = e(M/\pi M)$. Since $k$ is infinite, we can choose an ideal $\fa = (x_1, \cdots, x_{d-1}) \subseteq M = A/\fp$ such that $\fa$ is a parameter ideal for $M/\pi M$ and $e_\fa(M/\pi M) = e(M/\pi M)$. Let $\fb = \fa + (\pi)$. Since $\pi$ is $M$-regular and $\fb$ is a parameter ideal for $M$, we have \cite[Thm. 14.11]{Matsumura} that $e_\fb(M) = e_\fa(M/\pi M) = e(M/ \pi M) = e(M)$.  We consider the finite morphism of rings $R[[T_1, \cdots, T_{d-1}]] \to A/\fp = M$, defined via $T_i \mapsto x_i$, which, for dimensional reasons, must be injective. Arguing as in Case 2 of Proposition \ref{main_ineq}, we get an exact sequence
\[ 0 \to R[[T_1, \cdots, T_{d-1}]]^r \to M \to E \to 0 \]
with $\dim E < d = \dim M$, meaning that $e(M) = r = [K(A/\fp):K(R[[T_1, \cdots, T_n]])]$. Applying $- \oth_R N$ gives
\[ 0 \to N[[T_1, \cdots, T_{d-1}]]^r \to M \oth_R N \to E \oth_R N \to 0. \]
If we denote by $\fm'$ the maximal ideal of $N[[T_1, \cdots, T_{d-1}]] = (B/\fq)[[T_1, \cdots, T_{d-1}]]$, we see that
\[ e(M \oth_R N) \leq e_{\fm'}(M \oth_R N) = r e(N[[T_1, \cdots T_{d-1}]]) = e(M)e(N). \]
That $e(M \oth_R N)\geq e(M)e(N)$ is, of course, guaranteed by Theorem \ref{main_ineq}.

To prove the converse, we suppose that $e(M) < e(M/\pi M)$ and $e(N) < e(N/\pi N)$. Once again we choose a system of parameters $z_1, \cdots, z_d \in A/\fp$ such that $\fm_{A/\fp}$ is contained in the integral closure of $(z_1, \cdots, z_d)$. Mapping $W_i$ to $z_i$  produces a finite morphism $A' = R[[W_1, \cdots, W_d]] \to A/\fp$ whose kernel is generated by a prime $f$.  

We put $A_0 = A'/fA'$ and claim that $e(A_0) < e(A_0/ \pi A_0)$. For this, we see that as in Case 2 of Proposition \ref{main_ineq}, we have an exact sequence of $A_0$-modules
\[ 0 \to A_0^r \to M \to E \to 0 \]
with $\dim E < \dim M$, whence $e(M) = r e(A_0)$. We consider now the map on $\GG_0$, given by intersection with the divisor $\pi$:
\[\begin{array}{rcl}
\GG_0(\Spec A_0) & \stackrel{\Phi}{\longrightarrow} & \GG_0(\Spec (A_0/\pi A_0)) \\
\left[H\right]          & \longmapsto   & [H \otimes_{A_0} A_0/\pi A_0] - [\Tor_1^{A_0}(H, A_0/\pi A_0)]
\end{array} \]
Arguing as in the proof of Proposition \ref{decomp2}, we obtain, in $\GG_0(\Spec (A_0 / \pi A_0) )$, the relation
\[ [M/ \pi M] = \Phi([M]) = r \cdot \Phi([A_0]) + \Phi([E]) = r \cdot [A_0 / \pi A_0] + \Phi([E]) \]
where $\Phi([E])$ is represented by a sum of modules whose dimension is strictly less than $\dim M - 1$. Applying $e(\dim M - 1, -)$ shows that $e(M/\pi M) = r \cdot e(A_0 / \pi A_0)$, meaning that
\[e(M/\pi M) - e(M) = r(e(A_0/\pi A_0) - e(A_0)). \]
In particular, we obtain $e(A_0) < e(A_0/ \pi A_0)$.

Since $N$ is $R$-flat, we can apply $- \oth_R N$ to the above sequence and obtain $e(M \oth_R N) = r e(A_0 \oth_R N)$ (\emph{cf.} Case 2 of Proposition \ref{main_ineq}). Since $e(M) = r e(A_0)$, it therefore suffices to show that $e(A_0 \oth_R N) > e(A_0)e(N)$.

Consider first the shape of $f \in A'$. We may write
\[f = \sum_{\delta \in \mathbb{N}^d}{c_\delta W_1^{\delta_1}W_2^{\delta_2}\cdots W_d^{\delta_d}}\]
where $\mathbb{N} = \left\{0, 1, 2, \cdots \right\}$ and each $c_\delta \in R$. One therefore obtains the equality
\[ \ds e(A_0) = \ord(f) = \min_{\delta \in \mathbb{N}^d}\left\{\nu(c_\delta) + \delta_1 + \delta_2 + \cdots + \delta_d\right\} \]
where $\nu$ is the discrete valuation on $R$ and $\nu(0) = \infty$. If we denote by $\mathbf{f}$ the residue of $f$ in $A'/\pi A' = (R/\pi R)[[W_1, \cdots, W_d]]$, we see that
\[ e(A_0/\pi A_0) = \ord(\mathbf{f}) = \min_{\delta \in \mathbb{N}^d} \left\{\delta_1 + \cdots + \delta_d : \nu(c_\delta) = 0\right\}. \]

Since $e(A_0) < e(A_0/ \pi A_0)$, we see that $\nu(c_\delta) > 0$ for all monomials $c_\delta W_1^{\delta_1}\cdots W_d^{\delta_d}$ satisfying $\ord(c_\delta W_1^{\delta_1}\cdots W_d^{\delta_d}) = \ord(f)$. In other words, $f \equiv \pi g \, \operatorname{mod} \fm_{A'}^{a+1}$ where $a = \ord(f) = e(A_0)$ and $g \in A'$. Now put $B' = A' \oth_R (B/\fq)$. Clearly, $f \in \fm^a_{B'}$; denote by $\overline{f}$ its image in $\gr^a(B')=\fm^a_{B'}/\fm^{a+1}_{B'}$. It follows that we may write $\overline{f} = \opi \cdot \overline{g}$ in $\gr B'$ with $\opi$ regarded as an element of $\gr^1(B')$.

Since $B'$ is a power-series ring over $N = B/\fq$, we see that there is a canonical isomorphism $\gr B' \cong \gr(B/\fq)[\overline{W_1},\cdots \overline{W_d}]$ where each $\overline{W_i}$ sits in degree $1$. Since $\opi$ divides $\overline{f}$, we have a surjection
\[ \frac{\gr B'}{\overline{f} \gr B'} \twoheadrightarrow \frac{\gr B'}{\overline{\pi} \gr B'} \cong \left(\frac{\gr(B/\fq)}{\overline{\pi} \gr (B/\fq)}\right)[\overline{W_1}, \cdots, \overline{W_d}] \]
with $\opi$ now regarded as an element of $\gr^1(B/\fq)$. By assumption, $e(N) < e(N/\pi N)$, so Proposition \ref{ss_mult} assures us that $\ds \dim\left( \frac{\gr N}{\opi \gr N} \right) = \dim \gr N$. This, in turn, says that going modulo $\overline{f}$ does not reduce the dimension of $\gr B'$. Note that 
\[A_0 \oth_R N = (A'/fA') \oth_R (B/\fq) = B'/f B',\]
so we conclude, by Proposition \ref{ss_mult}, that $e(A_0 \oth_R N) > a e(B') = e(A_0)e(N)$.
\end{proof}

\begin{cor}\label{dimcut}Let $R$, $A$, $B$, $M$, and $N$ be as in Theorem \ref{ctp_eq}, and denote by $\opi$ the image of $\pi$ in $\gr^1(A)$ (or $\gr^1(B)$). Then
\[ \dim \left( \frac{\gr M \otimes_k \gr N}{\dopi} \right) = \dim (\gr(M \oth_R N)) \]
if and only if $e(M \oth_R N) = e(M)e(N)$.
\end{cor}
\begin{proof}Since we have assumed that $M$ and $N$ are $R$-flat, we know from Lemma \ref{ct_dim} that $\dim(\gr M \otimes_k \gr N) = \dim (\gr(M \oth_R N) ) + 1$. Our claim is therefore a question of whether $\dopi$ cuts down the the dimension of $\gr M \otimes_k \gr N$. As far as sets are concerned, we have
\[ \Supp(\gr M \otimes_k \gr N) = \bigcup_{i,j}{\Supp \left( \frac{\gr A}{\fp_i} \otimes_k \frac{\gr B}{\fq_j} \right)} \] 
where $\fp_i$ and $\fq_j$ range over the minimal primes of the graded modules $\gr M$ and $\gr N$.
Should it occur that $e(M \oth_R N) > e(M)e(N)$ then Theorem \ref{ctp_eq} implies that $e(M) < e(M/\pi M)$ and $e(N) < e(N/\pi N)$. This, in turn, means that $\opi$ fails to cut down the dimensions of both $\gr M$ and $\gr N$ by Proposition \ref{ss_mult}. Thus, $\opi$ lies in some $\ds \fp_i$ such that $\ds \dim\left(\frac{\gr A}{\fp_i}\right) = \dim (\gr M)$ as well as in some $\fq_j$ such that $\ds \dim\left( \frac{\gr B}{\fq_j} \right) = \dim (\gr N)$. It follows that going modulo $\dopi$ does not drop the dimension of $\gr M \otimes_k \gr N$.

Conversely, suppose that $e(M \oth_R N) = e(M)e(N)$. Then by Theorem \ref{ctp_eq} we may assume, without loss of generality, that $e(M) = e(M/\pi M)$ --- that is, $\opi \notin \fp_i$ for all $\fp_i$ such that $\ds \dim\left(\frac{\gr A}{\fp_i}\right) = \dim (\gr M)$. Now fix an arbitrary $\fq_j$. \\
\noindent\textbf{Case 1:} ($\opi \in \fq_j$) In this case, we see that
\[\frac{\gr(A)/\fp_i \otimes_k \gr(B)/\fq_j}{\dopi} = \frac{\gr(A)/\fp_i}{\opi} \otimes_k \gr(B)/\fq_j, \]
and it is clear that $\dopi$ cuts down the dimension of $\gr(A)/\fp_i \otimes_k \gr(B)/\fq_j$.
\noindent\textbf{Case 2:} ($\opi \notin \fq_j$) We consider going modulo the two-generated ideal $(\opi \otimes 1, 1 \otimes \opi)$:
\[\frac{\gr(A)/\fp_i \otimes_k \gr(B)/\fq_j}{(\opi \otimes 1 , 1 \otimes \opi)} = \frac{\gr(A)/\fp_i}{\opi} \otimes_k \frac{\gr(B)/\fq_j}{\opi}\]
Since going modulo $(\opi \otimes 1, 1 \otimes \opi) = (\dopi, \opi \otimes 1)$ drops the dimension by $2$, $\dopi$ must cut down the dimension by $1$.
\end{proof}

\subsection{Equidimensionality}
\begin{defns}Let $A$ be a Noetherian ring. A finitely-generated $A$-module $M$ is said to be \emph{equidimensional} if for all minimal primes $\fp \in \Supp M$, $\dim A/\fp = \dim M$. If $A$ is local with maximal ideal $\fm$ we say that $A$ is \emph{formally-equidimensional} if its $\fm$-adic completion $\hat{A}$ is equidimensional (as an $\hat{A}$-module).
\end{defns}
Formal equidimensionality is a crucial hypothesis for many of our arguments. We now show that this property is well-behaved with respect to completed tensor products.

\begin{prop}\label{equidim}Let $A = R[[X_1, \cdots, X_m]]$ and $B = R[[Y_1, \cdots, Y_n]]$. If $\fp$ and $\fq$ are prime ideals of $A$ and $B$ respectively such that $A/\fp$ and $B/\fq$ are $R$-flat, then $A/\fp \oth_R B/\fq$ is equidimensional with no embedded primes.
\end{prop}
\begin{proof}Since $A/\fp$ is $R$-flat, $\pi$ is a non-zerodivisor, so $A/\fp$ has a system of parameters $(\pi, x_1, \cdots, x_{d-1})$ where $d = \dim(A/\fp)$. By sending $T_i$ to $x_i$, we obtain a finite morphism $A' = R[[T_1, \cdots, T_{d-1}]] \to A/\fp$ which must be injective for dimensional reasons. Setting $r = [K(A/\fp):K(A')]$, we have an exact sequence of $A'$-modules
\[ 0 \to (A')^r \to A/\fp \to E \to 0 \]
where $E$ is annihilated by some $g \in A'$. Since $B/\fq$ is $R$-flat, we obtain
\[ 0 \to (A' \oth_R B/\fq)^r \to A/\fp \oth_R B/\fq \to E \oth_R B/\fq \to 0. \]
As $g$ also kills $E \oth_R B/\fq$, we see that $(A/\fp \oth_R B/\fq)_g \cong (A' \oth_R B/\fq)_g^r$ as $A' \oth_R B/\fq$-modules. Since $B/\fq$ is $R$-flat, $g$ remains a non-zerodivisor on $A/\fp \oth_R B/\fq$, thereby giving rise to an inclusion 
\[ A/\fp \oth_R B/\fq \hookrightarrow (A/\fp \oth_R B/\fq)_g \cong (A' \oth_R B/\fq)_g^r. \]
However, $A/\fp \oth_R B/\fq$ is finitely-generated over $A' \oth_R B/\fq$, so for some $N > 0$, one actually has an inclusion of $A' \oth_R B/\fq$-modules
\[ A/\fp \oth_R B/\fq \hookrightarrow \frac{1}{g^N}(A' \oth_R B/\fq)^r \cong (A' \oth_R B/\fq)^r. \]

Fix an associated prime $P$ of $A/\fp \oth_R B/\fq$. Its contraction $P' \in \Spec(A' \oth_R B/\fq)$ is also associated to $A/\fp \oth_R B/\fq$, now thought of as an $A' \oth_R B/\fq$-module. From the above inclusion, we see that $P'$ is necessarily associated to $A' \oth_R B/\fq$, itself. However, $A' \oth_R B/\fq \cong B/\fq[[T_1, \cdots, T_{d-1}]]$ and so is a domain, meaning that $P' = (0)$. We therefore see that $A' \oth_R B/\fq \to (A/\fp \oth_R B/\fq)/P$ is a finite, integral extension of rings, whence it follows that $\dim ((A/\fp \oth_R B/\fq)/P) = \dim(A' \oth_R B/\fq)$.
\end{proof}

\begin{cor}\label{equidim_cor}Let $A$ and $B$ be as in Proposition \ref{equidim}. Suppose that $M$ and $N$ are finitely-generated modules over $A$ and $B$ such that both are $R$-flat and equidimensional. Then $M \oth_R N$ is equidimensional.
\end{cor}
\begin{proof}As we saw in Lemma \ref{ct_dim}, we can decompose $\Supp(M \oth_R N)$ via
\[ \Supp(M \oth_R N) = \bigcup_{\fp, \fq}\Supp(A/\fp \oth_R B/\fq) \]
where $\fp$ and $\fq$ range over all minimal primes in $\Supp M$ and $\Supp N$. Since $M$ and $N$ are $R$-flat and equidimensional, we  have, again by Lemma \ref{ct_dim}, that for each $\fp$ and $\fq$,
\[\dim(A/\fp \oth_R B/\fq) = \dim(A/\fp) + \dim(B/\fq) - 1 = \dim M + \dim N - 1 = \dim(M\oth_RN). \]
We now appeal to Proposition \ref{equidim}.
\end{proof}

\begin{cor}\label{equichar_equidim}Let $A = K[[X_1, \cdots, X_m]]$ and $B = K[[Y_1, \cdots, Y_n]]$ with $K$ a field. If $M$ and $N$ are finitely-generated, equidimensional $A$ and $B$ modules, $M \oth_K N$ is equidimensional.
\end{cor}
\begin{proof}Put $R = K[[T]]$. Let $A' = A \oth_K R$ and $B' = B \oth_K R$. If $M' = M \oth_K R$ and $N' = N \oth_K R$, we see that $M'$ and $N'$ are $R$-flat, and, by Corollary \ref{equidim_cor}, $M' \oth_R N'$ is an equidimensional $A' \oth_R B'$-module on which $T$ acts as a non-zerodivisor. Thus,
\[ M \oth_K N = \frac{M' \oth_R N'}{T(M' \oth_R N')} \]
must be equidimensional \cite[B.4.4]{Huneke}.
\end{proof}

\subsection{Passage to the Associated Graded}
For power-series ring over a field $K$ one has, thanks to Samuel, a particularly nice interplay between the completed tensor product over $K$ and the associated graded:
\begin{thm}\label{gr_k}\cite{Samuel} Suppose $A = K[[X_1, \cdots, X_m]]$ and $B = K[[Y_1, \cdots, Y_n]]$ for $K$ a field. Then for ideals, $\fa$ and $\fb$, there is an isomorphism
\[ \gr(A/\fa) \otimes_K \gr(B/\fb) \stackrel{\sim}{\longrightarrow} \gr(A/\fa \oth_K B/\fb). \]
\end{thm}
From this we immediately see that $e(A/\fa \oth_K B/\fb) = e(A/\fa)e(B/\fb)$. In mixed-characteristic, the situation is considerably more subtle.

\begin{notns}
We now fix $A = R[[X_1, \cdots, X_m]]$ and $B = R[[Y_1, \cdots, Y_n]]$ where $R$ is our complete DVR with uniformizer $\pi$ and residue field $k$. Suppose that $A/\fa$ and $B/\fb$ are $R$-flat and equidimensional. As $A/\fa$, $B/\fb$, and $A/\fa \oth_R B/\fb$ are each equidimensional (Corollary \ref{equidim_cor}) and complete, we know \cite[B.4.6]{Huneke} that $\gr(A/\fa)$, $\gr(B/\fb)$, and $\gr(A/\fa \oth_R B/\fb)$ are all equidimensional as well. We shall tacitly use this fact repeatedly throughout our discussion.
\end{notns}

There is a canonical surjection of $\mathbb{N}$-graded rings
\[\gr(A/\fa) \otimes_k \gr(B/\fb) \twoheadrightarrow \gr(A/\fa \oth_R B/\fb) \]
where the grading on the left-hand side is given by total degree (see \ref{ntd} below). If we denote by $\opi$ the image of $\pi$ in $\gr^1 A$ (or  in $\gr^1 B$), we see that $\dopi$ maps to $0$ thus giving rise to a surjection
\[ \frac{\gr(A/\fa) \otimes_k \gr(B/\fb)}{\dopi} \stackrel{\psi}{\twoheadrightarrow} \gr(A/\fa \oth_R B/\fb). \]
Our situation is considerably worse than that in Theorem \ref{gr_k} as the kernel of $\psi$ can, in general, be quite large. In fact, Corollary \ref{dimcut} shows that the source and target of $\psi$ have the same dimension if and only if $e(A/\fa \oth_R B/\fb) = e(A/\fa)e(B/\fb)$.

\begin{notns}\label{ntd}We now put $M' = \left\{x \in \gr(A/\fa) : \opi^n \cdot x = 0 \mbox{ for some } n > 0\right\}$ and consider the exact sequence
\[0 \to M' \to \gr(A/\fa) \to M'' \to 0. \]
We now apply the functor $-\otimes_k \gr(B/\fb)$ and obtain
\begin{equation*}
\tag{*}0 \to M' \otimes_k \gr(B/\fb) \to \gr(A/\fa) \otimes_k \gr(B/\fb) \to M'' \otimes_k \gr(B/\fb) \to 0.
\end{equation*}
It will be convenient  to view this as an exact sequence of $\gr A \otimes_k \gr B $-modules with an $\mathbb{N}$-grading given by total degree: that is, the $r$-th graded piece of $\gr A \otimes_k \gr B$ is given by $\ds \bigoplus_{i+j=r}{\gr^i A \otimes_k \gr^j B}$. In this way, we can define Hilbert polynomials and therefore multiplicities for $\gr A \otimes_k \gr B$-modules as in Section 1.1.
\end{notns}

\begin{lem}\label{eq1}With the above notations, if $e(A/\fa) = e((A/\fa) / \pi (A/\fa))$, then 
\[ e(M'' \otimes_k \gr(B/\fb)) = e(\gr(A/\fa) \otimes_k \gr(B/\fb)) = e(A/\fa)e(B/\fb). \]
\end{lem}
\begin{proof}From Proposition \ref{ss_mult}, we see that no minimal prime of $\gr(A/\fa)$ contains $\opi$, and hence, $M'$ must vanish at every such prime, meaning that 
\[ \dim M' < \dim \gr(A/\fa) = \dim M''.\]
The result now follows from the exact sequence (*).
\end{proof}

\begin{lem}\label{nonhomogeneous}$\opi \otimes 1 - 1 \otimes \opi$ is a non-zerodivisor on $M'' \otimes \gr(B/\fb)$.
\end{lem}
\begin{proof}We can equip $S = \gr(A) \otimes_k \gr(B)$ with a new $\mathbb{N}$-grading that depends only on the first factor; that is, $S_r = \gr^r(A) \otimes_k \gr(B)$. We apply a similar grading to $N = M'' \otimes_k \gr(B/\fb)$. The key point is that $\opi \otimes 1 - 1 \otimes \opi$ is now nonhomogeneous and sits in degrees $0$ and $1$. Since $\opi \otimes 1$, the degree-$1$ component, is a non-zerodivisor on $M'' \otimes_k \gr(B/\fb)$, the result follows.
\end{proof}

\begin{cor}\label{eq2}If $e(A/\fa) = e((A/\fa)/\pi (A/\fa))$, then
\[ \ds e \left( \frac{M'' \otimes_k \gr(B/\fb)}{\dopi} \right) = e(A/\fa)e(B/\fb).\]
\end{cor}
\begin{proof}We have from Lemma \ref{nonhomogeneous} that $\dopi$ is a non-zerodivisor on $M'' \otimes \gr(B/\fb)$. Consider the exact sequence of $\gr A \otimes_k \gr B$ modules, $\mathbb{N}$-graded by total degree:
\[0 \to M'' \otimes_k \gr(B/\fb)[-1] \stackrel{\dopi}{\longrightarrow} M'' \otimes_k \gr(B/\fb) \to \frac{M'' \otimes_k \gr(B/\fb)}{\dopi} \to 0\]
Thus, the Hilbert function for $\ds \frac{M'' \otimes_k \gr(B/\fb)}{\dopi}$ is just the discrete derivative of that for $M'' \otimes_k \gr(B/\fb)$. From Lemma \ref{eq1}, we have $e(M'' \otimes_k \gr(B/\fb)) = e(A/\fa)e(B/\fb)$, and the conclusion follows.
\end{proof}

\begin{cor}\label{modopi}There is an exact sequence,
\[ 0 \to \frac{M' \otimes_k \gr(B/\fb)}{\dopi} \to \frac{\gr(A/\fa) \otimes \gr(B/\fb)}{\dopi} \to \frac{M'' \otimes_k \gr(B/\fb)}{\dopi} \to 0.\]
\end{cor}
\begin{proof}
The desired exact sequence is obtained by tensoring (*) with $\ds \frac{\gr(A) \otimes_k \gr(B)}{\dopi}$ and noting that
\[\Tor_1^{\gr(A)\otimes_k\gr(B)}\left(\frac{\gr(A)\otimes_k\gr(B)}{\opi \otimes 1 - 1 \otimes \opi}, M'' \otimes \gr(B/\fb)\right) = 0.\]
Indeed, since $\opi \otimes 1 - 1 \otimes \opi$ is clearly a non-zerodivisor in $\gr(A) \otimes_k \gr(B)$, we may identify this $\Tor$ term with the submodule of $M'' \otimes_k \gr(B/\fb)$ consisting of elements killed by $\dopi$, which, by Lemma \ref{nonhomogeneous}, is $\left\{0\right\}$.
\end{proof}

\begin{prop}\label{gr_R}Suppose that $e(A/\fa) = e((A/\fa)/\pi(A/\fa))$. Further suppose that any one of the following three conditions holds: 
\begin{itemize}
\item[(i)] $e(B/\fb) = e((B/\fb)/\pi(B/\fb))$; 
\item[(ii)] $\opi$ is not contained in any height-one associated prime of $\gr(A/\fa)$;
\item[(iii)] $\dim (A / \fa) = \dim A - 1$.
\end{itemize}
Then the surjection
\[ \frac{\gr(A/\fa) \otimes_k \gr(B/\fb)}{\dopi} \twoheadrightarrow \gr(A/\fa \oth_R B/\fb) \]
induces a homeomorphism on $\Spec$.
\end{prop}
\begin{proof}In the case of (i), $\opi$ cuts down the dimension of $\gr(B/\fb)$, so the same argument used in Corollary \ref{dimcut} can be applied to show that going modulo $\dopi$ reduces the dimension of $M' \otimes_k \gr(B/\fb)$ by one. Statement (ii) is equivalent to saying that $\dim M' \leq \dim \gr(A/\fa) - 2$. In either case, we see that
\[ \dim \left( \frac{M' \otimes_k \gr(B/\fb)}{\dopi} \right) < \dim \left( \frac{\gr(A/\fa) \otimes_k \gr(B/\fb)}{\dopi} \right). \]
From Lemma \ref{modopi}, we have the exact sequence
\[ 0 \to \frac{M' \otimes_k \gr(B/\fb)}{\dopi} \to \frac{\gr(A/\fa) \otimes \gr(B/\fb)}{\dopi} \to \frac{M'' \otimes_k \gr(B/\fb)}{\dopi} \to 0\]
from which we obtain
\[e \left( \frac{\gr(A/\fa) \otimes \gr(B/\fb)}{\dopi} \right) = e \left( \frac{M'' \otimes_k \gr(B/\fb)}{\dopi} \right) = e(A/\fa)e(B/\fb).\]
Note that the last equality follows from Lemma \ref{eq2}. Consider the exact sequence of $\mathbb{N}$-graded modules
\[ 0 \to K \to \frac{\gr(A/\fa) \otimes_k \gr(B/\fb)}{\dopi} \to \gr(A/\fa \oth_R B/\fb) \to 0. \]
The fact that the two right-hand terms have the same dimension (Corollary \ref{dimcut}) and multiplicity means that $K$ must have lesser dimension. Since $\ds \frac{\gr(A/\fa) \otimes_k \gr(B/\fb)}{\dopi}$ is equidimensional, this can only occur if 
\[\Supp \left(\frac{\gr(A/\fa) \otimes_k \gr(B/\fb)}{\dopi}\right) = \Supp \left( \gr(A/\fa \oth_R B/\fb) \right) \]
as claimed.

For case (iii), we first remark that we have
\[ \Supp(\gr(A/\fa)) = \bigcup_i{\Supp(\gr(A/\fp_i))} \]
and
\[ \Supp(\gr(A/\fa \oth_R B/\fb)) = \bigcup_i{\Supp(\gr(A/\fp_i \oth_R B/\fb))} \]
with $\fp_i$ ranging over all minimal primes of $A/\fa$. It therefore suffices to prove the claim for each surjection
\[ \frac{\gr(A/\fp_i) \otimes \gr(B/\fb)}{\dopi} \twoheadrightarrow \gr(A/\fp_i \oth_R B/\fb). \]
Since $A$ is a regular local ring, the fact that each $A/\fp_i$ has dimension $\dim A - 1$ means that $\operatorname{ht}(\fp_i) = 1$, whence we have $\fp_i = (f_i)$ for some $f_i \in A$. One therefore has an isomorphism $\gr(A/\fp_i) \cong \gr(A)/(\overline{f_i})$ (see Lemma \ref{gr_domain} in the next section). From Proposition \ref{decomp2}, $e(A/\fp_i) = e((A/\fp_i)/\pi(A/\fp_i))$, meaning that $\opi$, now thought of as a prime element in $\gr^1(A)$, does not lie in any minimal prime of $\gr(A/\fp_i) = \gr(A)/(\overline{f_i})$ (Proposition \ref{ss_mult}). Since $\gr A$ is a unique factorization domain, this is enough to guarantee that $\opi$ is a non-zerodivisor on $\gr(A/\fp_i)$ and, in particular, is not contained in any associated prime of $\gr(A/\fp_i)$. We are therefore reduced to the situation of Case (ii) for which we already know the result to be true.
\end{proof}

\begin{exmp}It is very possible for $\opi$ to lie in a height-one associated prime while, at the same time, missing every minimal one. For example, let $R = \mathbb{C}[[T]]$ be our DVR and $A = R[[X,Y]] = \mathbb{C}[[T,X,Y]]$. Let $\fp$ be the kernel of the morphism $A \to \mathbb{C}[[Z]]$ under the mapping $T \mapsto Z^4$, $X \mapsto Z^5$, and $Y \mapsto Z^{11}$. $\fp$ contains the relations, $T^4-XY$, $X^3-TY$, $X^4-T^5$, and $Y^2 - T^3X^2$. From this, one sees that the integral closure of $(T)$ in $A/\fp$ is the entire maximal ideal. Thus,
\[e(A/\fp) = e_{(T)}(A/\fp) = e((A/\fp)/T(A/\fp)),\]
where the first equality follows from Proposition \ref{rees_easy} and the second by \cite[Thm. 14.11]{Matsumura}. By Proposition \ref{ss_mult}, this is enough to guarantee that $\overline{T}$ lies outside each minimal prime of $\gr(A/\fp)$.  However, the augmentation ideal $(\overline{T},\overline{X},\overline{Y})$ of $\gr(A/\fp)$ annihilates $\overline{Y}$ and thus is associated.
\end{exmp}
\section{Serre Intersection Multiplicity}
\subsection{The Lower Bound in the Unramified Case} With the results we have developed in Section 2, we now prove Theorems A through E stated in the introduction. Let us first succinctly summarize the work of Serre:

\begin{thm}\label{serre_dvr}\cite[V-13,16]{Serre} Suppose that $V$ is \textbf{either} a field or a complete DVR. Let $A = V[[X_1, \cdots, X_n]]$ and suppose that $M$ and $N$ are finitely-generated modules such that $\ell(M \otimes_A N) < \infty$. Then the following statements hold:
\begin{itemize}
\item[(a)] $\ds \chi^A(M,N) := \sum_{i=0}^{\dim A}{(-1)^i \ell(\Tor_i^A(M,N)) \geq 0}$.
\item[(b)] $\dim M + \dim N \leq \dim A$.
\item[(c)] $\chi^A(M,N) > 0$ if and only if $\dim M + \dim N = \dim A$.
\end{itemize}
\end{thm}

\begin{lem}\label{smooth}Suppose that $(R_0,\pi R_0)$ is a discrete valuation ring and that $R_0/\pi R_0$ is perfect. Let $(A,\fm)$ be a regular local ring containing $R_0$ such that $\pi \in \fm-\fm^2$. Then $\hat{A} \cong R[[X_1, \cdots, X_n]]$ where $(R,\pi R)$ is complete and contains $R_0$.
\end{lem}
\begin{proof}The proof follows the same outline as that of the Cohen Structure Theorem found in \cite[Ch. 29]{Matsumura}. We shall be brief. Suppose that $(\pi, x_1, \cdots, x_n)$ is a regular system of parameters for $A$. If $R = \hat{A}/(x_1, \cdots, x_n)\hat{A}$, then $R$ is a flat extension of $R_0$ and has $\pi$ as its uniformizer. Since $R_0/\pi R_0 \to R/\pi R$ is separable (i.e. $0$-smooth), it follows \cite[28.10]{Matsumura} that that $R_0 \to R$ is $\pi R$-smooth. In particular, if we inductively assume \hspace{1mm}that $R \to A/\fm^i$ has been defined, we obtain a lift as indicated by the dotted arrow:
\[\xymatrix{
R \ar[r] \ar@{.>}[rd] & A/\fm^{i}  \\
R_0 \ar[u] \ar[r] & A/\fm^{i+1} \ar[u] \\
}\]
Ultimately, we see that the map $R_0 \to \hat{A}$ factors through $R$. We can then define $R[[X_1, \cdots, X_n]] \to \hat{A}$ via $X_i \mapsto x_i$. Since $R/\pi R = \hat{A}/\fm\hat{A}$, the map is clearly surjective and, for dimensional reasons, must have kernel equal to $\left\{0\right\}$.
\end{proof}

\begin{mthm}\label{unramified_ineq} Let $R_0$ and $A$ be as in Lemma \ref{smooth}. If $M$ and $N$ are finitely-generated $A$-modules such that $\dim M + \dim N = \dim A$ and $\ell(M \otimes_A N) < \infty$, then $\chi^A(M,N) \geq e(M)e(N)$.
\end{mthm}
\begin{proof}Following \cite[V-16]{Serre} we begin with two reductions. First, under the faithfully-flat base-change to $\hat{A}$, our three numbers $e(M)$, $e(N)$, and $\chi(M,N)$ are unchanged, so we shall henceforth assume that $A = \hat{A} = R[[X_1, \cdots, X_n]]$ where $R$ is a complete DVR with uniformizer $\pi$ (Lemma \ref{smooth}). Using the bilinearity of $\chi^A(-,-)$, we can reduce to the case where $M = A/\fp$ and $N = A/\fq$ for primes $\fp$ and $\fq$, just as in proofs of Theorems \ref{main_ineq} and \ref{ctp_eq}.

Note first that it cannot occur that $\pi \in \fp$ and $\pi \in \fq$. If this were the case, both $M$ and $N$ would be supported on $(R/\pi R)[[X_1, \cdots, X_n]]$, so by Theorem \ref{serre_dvr}(b), we would have $\dim M + \dim N \leq \dim (R/\pi R)[[X_1, \cdots, X_n]] = \dim A - 1$. Thus, $\pi$ is a non-zerodivisor for one of the modules --- say $N$. Consider the diagonal ideal, $\fd \subseteq A \oth_R A$, that is generated by the elements, $\left\{X_i \oth 1 - 1 \oth X_i\right\}_{1 \leq i \leq n}$. There is a ``reduction to the diagonal'' spectral sequence \cite[V-12]{Serre}
\[ E^2_{pq} = \Tor_p^{A \oth_R A}((A \oth_R A)/\fd, \widehat{\Tor_q^R}(M,N)) \Rightarrow \Tor_{p+q}^A(M,N). \]
Since $N$ is $R$-flat, the spectral sequence degenerates to give isomorphisms
\[ \Tor^{A \oth_R A}_p((A \oth_R A)/\fd,M \oth_R N) \cong \Tor_p^A(M,N) \]
for all $p \geq 0$. We therefore obtain
\[ \begin{array}{rcl}\chi^A(M,N) & = & \ds \sum_{i=0}^{n}{(-1)^i\ell(\Tor_i^{A\oth_R A}((A \oth_R A)/\fd, M \oth_R N))} \\
 & = & \ds \sum_{i=0}^{n}{(-1)^i \ell(H_i(\fd, M \oth_R N))} \end{array} \]
where $H_*(\fd, M \oth_R N)$ is the Koszul homology with respect to the parameter system $\left\{X_i \oth 1 - 1 \oth X_i\right\}_{1 \leq i \leq n}$. 
From \cite[IV-12]{Serre}, we see that $\chi^A(M,N) = e_\fd(M \oth_R N)$, so by Theorem \ref{main_ineq}, we conclude that $e_\fd(M \oth_R N) \geq e(M \oth_R N) \geq e(M)e(N)$.
\end{proof}

\subsection{Connection with the Blowup} 
\begin{defns}\cite[B.6.9]{Fulton} Let $T \hookrightarrow Y \hookrightarrow X$ be a sequence of closed immersions of (Noetherian) schemes. Denote by $\phi_X:\widetilde{X} \to X$ and $\phi_Y:\widetilde{Y} \to Y$ the blowups of $X$ and $Y$ along $T$. $\widetilde{Y}$ may be realized as the scheme-theoretic closure of $\phi_X^{-1}(Y-T)$ in $\widetilde{X}$ and is often called \emph{strict transform of $Y$ in $\widetilde{X}$}. Furthermore, if we denote by $E_X = \phi_X^{-1}(T)$ and $E_Y=\phi_Y^{-1}(T)$ the exceptional divisors of our blowups, one has $E_Y = \widetilde{Y} \cap E_X := \widetilde{Y} \times_{\widetilde{X}} E_X$.
\end{defns}

Now let $(A,\fm)$ be a regular local ring. We shall mainly be concerned with the blowup of $X = \Spec A$ along $T = \Spec A/\fm$. We now recall from the introduction the following key formula:
\begin{thm}\label{blowupformula}\cite[Example 20.4.3]{Fulton} Let $X = \Spec A$ where $A$ is a regular local ring. Consider the closed subschemes $Y = \Spec (A/\fp)$ and $Z = \Spec(A/\fq)$ with $\fp$ and $\fq$ prime ideals of $A$. Suppose that $\dim A/\fp + \dim A/\fq = \dim A$, and that $\ell(A/\fp \otimes_A A/\fq) < \infty$. Then 
\[ \chi^A(A/\fp,A/\fq) = e(A/\fp)e(A/\fq) + \chi^{\mco_{\widetilde{X}}}(\mco_{\widetilde{Y}},\mco_{\widetilde{Z}}) \]
where $\widetilde{X}$, $\widetilde{Y}$, and $\widetilde{Z}$ are the blowups of $X$, $Y$, and $Z$ at the closed point and
\[ \chi^{\mco_{\widetilde{X}}} (\mco_{\widetilde{Y}}, \mco_{\widetilde{Z}}) := \sum_{i,j \geq 0}{ (-1)^{i+j} \operatorname{length}_A (\bH^i(\widetilde{X},(\Tor_j^{\mco_{\widetilde{X}}}(\mco_{\widetilde{Y}},\mco_{\widetilde{Z}}))))}. \]
\end{thm}
A proof using algebraic cycles is outlined in \cite{Fulton} while a more direct approach may be found in \cite{dutta1}.

When $A$ satisfies the hypotheses of Theorem \ref{unramified_ineq}, we see that $\chi^{\mco_{\widetilde{X}}}(\mco_{\widetilde{Y}},\mco_{\widetilde{Z}}) \geq 0$. In this case, proving Conjecture I would amount to showing that if $\chi^{\mco_{\widetilde{X}}}(\mco_{\widetilde{Y}},\mco_{\widetilde{Z}}) = 0$ (that is, if $\chi^A(A/\fp,A/\fq) = e(A/\fp)e(A/\fq)$), then $\widetilde{Y} \cap \widetilde{Z}$ is empty. Before proceeding further, we first give algebraic descriptions of $\widetilde{Y} \cap \widetilde{Z}$ and several other geometric objects relating to the blowup.

\begin{lem}\label{st_explained}Let $X$ be the spectrum of a Noetherian local ring $A$, and let \newline $Y = \Spec(A/\fa)$ and $Z = \Spec(A/\fb)$ be arbitrary closed subschemes. Denote by $\widetilde{X}, \widetilde{Y}$, and $\widetilde{Z}$ the blowups at the closed point. There is an isomorphism of reduced schemes
\[(\widetilde{Y} \cap \widetilde{Z})_{\red} \cong \Proj(\gr(A/\fa) \otimes_{\gr A} \gr(A/\fb))_{\red}. \] 
In particular, $\widetilde{Y} \cap \widetilde{Z}$ is empty if and only if $\gr(A/\fa) \otimes_{\gr A} \gr (A/\fb)$ is $0$-dimensional.
\end{lem}
\begin{proof}Denote by $E = \Proj(\gr A)$ the exceptional divisor of the blowup $\widetilde{X} \to X$. We have that $E \cap \widetilde{Y} = \Proj(\gr(A/\fa))$ and $E \cap \widetilde{Z} = \Proj(\gr(A/\fb))$. Since $Y$ and $Z$ meet at a point in $X$, we clearly have the \emph{set-theoretic} containment $\tilde{Y} \cap \tilde{Z} \subseteq E$. Thus, as sets, we have $\widetilde{Y} \cap \widetilde{Z} = (E \cap \tilde{Y}) \cap (E \cap \widetilde{Z})$. That is
\[ (\widetilde{Y} \cap \widetilde{Z})_{\red} = ((\widetilde{Y} \cap E) \times_E (\widetilde{Z} \cap E))_{\red} = \Proj(\gr(A/\fa) \otimes_{\gr A} \gr(A/\fb))_{\red}. \]
\end{proof}

\begin{lem}\label{gr_domain}\cite[Exercise 14.2]{Matsumura} Let $(A,\fm)$ be a Noetherian local ring for which $\gr A$ is an integral domain. Then if $f \in \fm^n - \fm^{n+1}$ with initial form $\overline{f} \in \fm^n/\fm^{n+1}$ there is an isomorphism
\[\frac{\gr A}{(\overline{f})} \stackrel{\sim}{\longrightarrow} \gr (A/fA). \]
\end{lem}

This isomorphism may not hold if $\gr A$ fails to be a domain. For example, Proposition \ref{ss_mult} shows that in general, we should not even expect both sides to have the same dimension. Nonetheless, we have the following partial result which will suffice for our purposes.
\begin{prop}\label{redproj}Let $(A,\fm)$ be a local ring and suppose that $f \in A$ is a non-zerodivisor for which $e(A/fA) = e(A)$. Then the surjection $\gr A/(\overline{f}) \to \gr (A/fA)$ induces an isomorphism of reduced schemes:
\[\Proj(\gr (A/fA))_{\red} \stackrel{\sim}{\longrightarrow} \Proj \left( \frac{\gr A}{(\overline{f})} \right)_{\red} \]
\end{prop}
\begin{proof}Let $\phi: \widetilde{X} \to \Spec A$ be the blowup along the closed point with exceptional divisor $E = \phi^{-1}(\Spec A/\fm)$. Let $\widetilde{Y}$ be the strict transform of $\Spec(A/fA)$. We can identify $\Proj(\gr (A/fA))$ with $\widetilde{Y} \cap E$. The canonical surjection of graded rings gives a closed immersion of $\Proj(\gr(A/fA))$ into $\Proj(\gr(A)/(\overline{f}))$, which, itself, is a subscheme of $E$. To see that this inclusion is, in fact, a bijection of sets, choose any component of $\Proj(\gr(A)/(\overline{f}))$ and let $x \in \widetilde{X}$ be its generic point. It will suffice to show that $\Proj(\gr(A/fA)) = \widetilde{Y} \cap E$ contains $x$. Since $x$ already belongs to $E$, we need only show that $x \in \widetilde{Y}$. 

Let $B = \mco_{\widetilde{X},x}$. We consider now the local equations of $\Proj(\gr A/(\overline{f}))$ and $\Proj(\gr(A/fA))$ in $B$. First, $E = \Proj(\gr A)$ is an effective Cartier divisor and so is given by a single equation $t \in B$. As $f \in \fm-\fm^2$, we obtain the factorization $f = at$ in $B$, and $\Proj(\gr A/(\overline{f}))$ is given locally in $B$ by the ideal $(a,t)$. We may also realize $\widetilde{Y}$ as the closure of $\phi^{-1}(\Spec(A/fA)) - E$ in $\widetilde{X}$, so locally in $\Spec B$, $\widetilde{Y}$ is given via the closure of the locally-closed subset, $\Spec ((B/fB)_t) = \Spec ((B/aB)_t)$. Now, $\Spec ((B/aB)_t)$ is empty if and only if the vanishing set of $a$ in $\Spec B$ is contained in the vanishing set of $t$. However, from Proposition \ref{ss_mult}, we know that since $\overline{f}$ cuts down the dimension of $\gr A$, we must have that $\dim (B/(a,t)) < \dim(B/(t))$. That is, the vanishing sets of $t$ and $a$ meet properly. Thus, $\Spec (B/aB)_t$ is nonempty and so its closure contains the maximal ideal of $B$, meaning that $x \in \widetilde{Y}$.
\end{proof}

\begin{lem}\label{blowup_dim}Let $(A,\fm)$ be a formally-equidimensional, universally catenary local ring. Let $\widetilde{X}$ be the blowup along the maximal ideal. Then for every closed point $x \in \widetilde{X}$, $\dim \mco_{\widetilde{X},x} = \dim A$.
\end{lem}
\begin{proof}First, we observe that if $x \in \widetilde{X}$ is closed, then since $\widetilde{X} \to \Spec A$ is proper, $x$ must be mapped to the closed point; hence, $x \in E = \Proj(\gr A)$. Since $A$ is formally-equidimensional and universally catenary, we have that $\gr(A)$ is an equidimensional ring \cite[B.4.6]{Huneke}. Thus, $E$ is an equidimensional projective scheme over the field $k = A/\fm$, so for any closed $x \in \widetilde{X}$, $\dim \mco_{E,x} = \dim \gr A - 1 = \dim A -1$. But $\mco_{E,x}$ is just $\mco_{\widetilde{X},x}/t\mco_{\widetilde{X},x}$ where $t$ is the local equation for the effective Cartier divisor $E$, so $\dim \mco_{\widetilde{X},x} = \dim A$.
\end{proof}

\begin{lem}\label{unramified_locus}Suppose that $(R_0, \pi R_0)$ is a discrete valuation ring with perfect residue field. Let $(A,\fm)$ be a regular local ring containing $R_0$ such that $\pi \in \fm - \fm^2$. Denote by $\widetilde{X}$ the blowup of $\Spec A$ at the maximal ideal with exceptional divisor $E = \Proj (\gr A) = \mathbb{P}^d_k$. Put $W = \Proj(\gr A/\overline{\pi}\gr A) \hookrightarrow E$. Then for every $x \in E - W$, $\mco_{\widetilde{X},x}/\pi \mco_{\widetilde{X},x}$ is regular. In particular, $\widehat{\mco}_{\widetilde{X},x}$ is a power-series over a complete discrete valuation ring.
\end{lem}
\begin{proof}
If $A$ has dimension $d+1$, then we can write $\fm = (\pi, y_1, \cdots, y_d)$. $X$ is covered by open, affine coordinate charts $X_i = \Spec A_i$ where
\[ A_0 = A\left[\frac{y_1}{\pi}, \frac{y_2}{\pi},  \cdots, \frac{y_d}{\pi}\right] \mbox{ and } A_i = A\left[\frac{\pi}{y_i}, \frac{y_1}{y_i}, \cdots, \frac{y_{i-1}}{y_i}, \frac{y_{i+1}}{y_i}, \cdots, \frac{y_d}{y_i} \right] \hspace{1mm} (1 \leq i \leq d). \]
On $X_0$, $\pi$ generates $E$, so the result is clear for $x \in X_0 \cap E$. For $i \geq 1$, $E \cap X_i$ is cut out by $y_i$ and $W$ by $(y_i, \pi/y_i)$. Thus, $x \in X_i \cap (E - W)$ corresponds to a prime ideal $\fp \subseteq A_i$ containing $y_i$ but not $\pi/y_i$. In $\mco_{X,x} = (A_i)_\fp$, $\pi$ and $y_i$ therefore differ by the unit $\pi/y_i$, so once again, $\mco_{\widetilde{X},x}/\pi \mco_{\widetilde{X},x} \cong \mco_{E,x}$ is regular. Since $R_0 \subseteq \widehat{\mco}_{\widetilde{X},x}$, the second claim follows from Lemma \ref{smooth}.
\end{proof}

With $A$ and $R_0$ as above, let $Y = \Spec(A/\fp)$ and $Z = \Spec(A/\fq)$ be as in Theorem \ref{blowupformula}. Our next proposition shows that if $\chi^A(A/\fp,A/\fq) = e(A/\fp)e(A/\fq)$, then $\widetilde{Y} \cap \widetilde{Z}$ is contained (as a set) inside the closed, codimension-$2$ subscheme $W$ defined in Lemma \ref{unramified_locus}.

\begin{prop}\label{good_locus}Let $R_0$ and $A$ be as in Lemma \ref{unramified_locus}. Let $\fp$ and $\fq$ be prime ideals of $A$ for which $\dim(A/\fp) + \dim(A/\fq) = \dim A$ and $\ell(A/\fp \otimes_A A/\fq) < \infty$ Suppose that $\chi^A(A/\fp,A/\fq) = e(A/\fp)e(A/\fq)$. Under these conditions, if
\[ \dim \left( \gr(A/\fp) \otimes_{\gr A} \gr(A/\fq) \otimes_{\gr A} \frac{\gr A}{\opi \gr A} \right) = 0, \]
then $\dim ( \gr(A/\fp) \otimes_{\gr A} \gr(A/\fq) ) = 0$.
\end{prop}
\begin{proof}
We put $X = \Spec A$, $Y = \Spec (A/\fp)$, $Z = \Spec (A/\fq)$ and denote by $\widetilde{X}$, $\widetilde{Y}$, and $\widetilde{Z}$ their respective blowups at the closed point. We let $E$ be the exceptional divisor of $\widetilde{X}$ and consider the closed subscheme $W = \Proj(\gr A/(\opi \gr A)) \hookrightarrow E$ as defined in Lemma \ref{unramified_locus}. From the hypothesis, $\dim (\gr(A/\fp) \otimes_{\gr A} \gr(A/\fq)) \leq 1$; that is, $\widetilde{Y} \cap \widetilde{Z}$ is a finite (perhaps empty) set of closed points $x_1, \cdots, x_n \in E$. We also know that $\widetilde{Y} \cap \widetilde{Z} \cap W$ is empty, so according to Lemma \ref{unramified_locus}, $\widehat{\mco}_{\widetilde{X},x_i}$ is a power-series over a complete DVR for each of the $x_i$.
Each of the coherent sheaves $\Tor_j^{\mco_{\widetilde{X}}}(\mco_{\widetilde{Y}},\mco_{\widetilde{Z}})$ is supported on (at most) finitely many closed points and so is $\Gamma(\widetilde{X},-)$-acyclic. We therefore see \cite[Main Thm.,(iii)]{Dutta_Blowup} that the formula for $\chi^{\mco_{\widetilde{X}}}(\mco_{\widetilde{Y}},\mco_{\widetilde{Z}})$ reduces to
\[\begin{array}{rcl} \chi^{\mco_{\widetilde{X}}}(\mco_{\widetilde{Y}},\mco_{\widetilde{Z}}) & = &
\ds \sum_{i=1}^{n}{\sum_{j \geq 0}{(-1)^j\operatorname{length}_A(\Tor^{\mco_{\widetilde{X},x_i}}_j(\mco_{\widetilde{Y},x_i}, \mco_{\widetilde{Z},x_i}))}} \\
 & = &
\ds \sum_{i=1}^{n}{[k(x_i):k]\chi^{\mco_{\widetilde{X},x_i}}(\mco_{\widetilde{Y},x_i}, \mco_{\widetilde{Z},x_i})} \end{array} \]
where we respectively denote by $k(x_i)$ and $k$ the residue fields of $\mco_{\widetilde{X},x_i}$ and $A$.
As $A/\fp$ and $A/\fq$ are quotients of a regular local ring and therefore satisfy the hypotheses of Lemma \ref{blowup_dim}, we see that $\dim \mco_{\widetilde{Y},x_i} + \dim \mco_{\widetilde{Z},x_i} = \dim \mco_{\widetilde{X},x_i}$. Hence, from Theorem \ref{unramified_ineq}, we have $\chi^{\mco_{\widetilde{X}}}(\mco_{\widetilde{Y}},\mco_{\widetilde{Z}}) > 0$ if and only if $\widetilde{Y} \cap \widetilde{Z}$ is nonempty. On the other hand, since $\chi^{A}(A/\fp,A/\fp) = e(A/\fp)e(A/\fq)$, Theorem \ref{blowupformula} says that $\chi^{\mco_{\widetilde{X}}}(\mco_{\widetilde{Y}},\mco_{\widetilde{Z}}) = 0$, and the conclusion follows by Lemma \ref{st_explained}.
\end{proof}

\subsection{Equality in the Equicharacteristic Case}
We begin by stating the following theorem of Fulton-Lazarsfeld, slightly reworded in our language:
\begin{thm}\label{fulton_equichar}\cite[Thm. 12.4(b)]{Fulton} Let $V$ be a smooth variety over a field $k$. Suppose that $x \in V$ is a $k$-rational point and that $W_1$ and $W_2$ are subvarieties of complimentary dimension meeting properly at $x$. Let $A = \mco_{V,x}$, $M = \mco_{W_1,x}$, and $N = \mco_{W_2, x}$. Then $\chi^A(M,N) \geq e(M)e(N)$ with equality occurring if and only if the strict transforms of the $W_i$ do not meet on $\widetilde{V}$, the blowup at $\left\{x\right\}$.
\end{thm}

Indeed, a routine application of Artin Approximation \cite{Artin} --- as seen in \cite[\S 3]{Dutta_Bass} or \cite[Appx.]{Hochster} --- could very likely be used to reduce the problem for a general equicharacteristic $A$ to the finite-type situation of Theorem \ref{fulton_equichar}. However, since Fulton's proof requires the smoothness of the blowup over the base, it appears difficult to adapt it to the case of mixed-characteristic. Instead, we shall now present a direct, algebraic argument which generalizes Theorem \ref{fulton_equichar} to all equicharacteristic regular local rings and applies, at least in part, to the case of mixed-characteristic. A key ingredient is the Theorem of Rees (\ref{rees_thm}).

\begin{mthm}\label{fulton_blowup}Suppose that $A$ is an equicharacteristic regular local ring with residue field $k$. Suppose that $M$ and $N$ are two equidimensional modules satisfying both $\dim M + \dim N = \dim A$ and $\ell(M \otimes_A N) < \infty$. If $\chi^A(M,N) = e(M)e(N)$, then $\dim(\gr M \otimes_{\gr A} \gr N) = 0$. In other words, Conjecture I holds in the equicharacteristic case.
\end{mthm}
\begin{proof}By base-extending to $\hat{A}$, it is clear that the intersection multiplicity and the dimension of $\gr M \otimes_{\gr A} \gr N$ will remain unchanged. Since $A$ is regular and hence formally-equidimensional, we are guaranteed \cite[B.4.2,3]{Huneke} that $\widehat{M}$ and $\widehat{N}$ are equidimensional. We may therefore assume that $A = \hat{A} = k[[X_1, \cdots, X_n]]$. If we choose arbitrary minimal primes $\fp$ of $M$ and $\fq$ of $N$, we necessarily have $\chi^{{A}}({A}/\fp,{A}/\fq) = e(A/\fp)e(A/\fq)$. It therefore suffices to prove the proposition for $M = A/\fp$ and $N = A/\fq$.

By tracing through the steps of the proof of Theorem \ref{serre_dvr} (or adapting the proof of Theorem A), we see that
\[ \chi(A/\fp,A/\fq) = e_\fd(A/\fp \oth_k A/\fq)\]
where $\fd \subseteq A \oth_k A$ is generated by $\left\{X_i \oth 1 - 1 \oth X_i\right\}_{(1 \leq i \leq n)}$. From Theorem \ref{gr_k}, we have an isomorphism
\[  \gr(A/\fp) \otimes_k \gr(A/\fq) \stackrel{\sim}{\longrightarrow}  \gr(A/\fp \oth_k A/\fq)\]
under which $\overline{X_i \oth 1 - 1 \oth X_i} \in \gr^1(A/\fp \oth_k A/\fq)$ corresponds to $\overline{X_i} \otimes 1 - 1 \otimes \overline{X_i}$.
From this isomorphism, we immediately see that $e(A/\fp \oth_k A/\fq) = e(A/\fp)e(A/\fq)$. By hypothesis, however, we also have $\chi(A/\fp,A/\fq) = e_\fd(A/\fp \oth_k A/\fq) = e(A/\fp)e(A/\fq)$. Since $A/\fp \oth_k A/\fq$ is equidimensional (Corollary \ref{equichar_equidim}), Rees's Theorem (\ref{rees_thm}) says that the maximal ideal of $A/\fp \oth_k A/\fq$ is contained in the integral closure of $\fd\cdot(A/\fp \oth_k A/\fq)$. The map of Rees Algebras
\[ (A/\fp \oth_k A/\fq)[\fd T] \hookrightarrow (A/\fp \oth_k A/\fq)[\fm T] \]
is therefore finite \cite[Thm. 8.2.1]{Huneke}; and hence, the same can be said for $\gr_\fd(A/\fp \oth_k A/\fq) \to \gr(A/\fp \oth_k A/\fq)$. It's therefore clear that
\[ \frac{\gr(A/\fp \oth_k A/\fq)}{(\cdots, \overline{X_i \otimes 1 - 1 \otimes X_i}, \cdots)} \cong \frac{\gr(A/\fp) \otimes_k \gr(A/\fq)}{(\cdots, \overline{X_i} \otimes 1 - 1 \otimes \overline{X_i}, \cdots)} \cong \gr(A/\fp) \otimes_{\gr A} \gr(A/\fq)\]
is $0$-dimensional as claimed.
\end{proof}

\begin{exmp}\label{exmp_equidim}To illustrate the necessity of the equidimensional hypotheses, let $A = \mathbb{C}[[X,Y,Z]]$, $I = (X(Y-X^2), XZ)$ and $J = (Y,Z)$. By considering the exact sequence
\[0 \to  A/(Y - X^2, Z) \stackrel{X}{\rightarrow} A/I \to A/(X) \to 0 \]
and noting that $\dim (A/(Y - X^2,Z)) < \dim A/I$, we have 
\[ \chi^A(A/I,A/J) = \chi^A(A/(X),A/J) = 1 = e(A/I)e(A/J). \]
A direct computation, however, shows that $\dim(\gr(A/I) \otimes_{\gr A} \gr(A/J)) = 1$.
\end{exmp}

\subsection{Equality in the Mixed-Characteristic Case}
We now fix a discrete valuation ring $(R_0,\pi R_0)$ with perfect residue field. We suppose that $(A,\fm)$ is a regular local ring containing $R_0$ and that $\pi \in \fm - \fm^2$. This is enough to guarantee (Lemma \ref{smooth}) that $\hat{A} \cong R[[X_1, \cdots X_n]]$ for some complete DVR $(R,\pi R)$. We treat first the case in which one of the modules is supported on $A/\pi A$, and hence, the situation is considerably simpler: 

\begin{mthm}With $A$ as above, let $M$ and $N$ be finitely-generated, equidimensional $A$-modules satisfying $\dim M + \dim N = \dim A$ and $\ell(M \otimes_A N) < \infty$. Suppose also that $\Supp M \subseteq \Supp (A/\pi A)$. If $\chi^A(M,N) = e(M)e(N)$, then $\gr(M) \otimes_{\gr A} \gr(N)$ is $0$-dimensional. In other words, Conjecture I holds in this case.
\end{mthm}
\begin{proof}
Arguing as in the proof of Theorem B, we may assume without loss of generality that $A = \hat{A} = R[[X_1, \cdots, X_n]]$ where $R$ is a complete DVR with uniformizer $\pi$. We may also reduce to the case where $M = A/\fp$ and $N=A/\fq$ for prime ideals $\fp$ and $\fq$. The hypothesis on $\Supp M$ now amounts to saying that $\pi M = 0$, so once again by Theorem \ref{serre_dvr}(b), $\pi$ is a non-zerodivisor on $N$.

Let $\overline{A} = A/\pi A = k[[X_1, \cdots, X_n]]$ and $\overline{A/\fq} = (A/\fq)/\pi(A/\fq)$. Since $\pi$ is a non-zerodivisor on $N = A/\fq$, $\Tor_p^A(\overline{A},A/\fq) =0$ for $p>0$. The change of rings spectral sequence
\[ E^2_{pq} = \Tor_p^{\overline{A}}(A/\fp, \Tor_q^A(\overline{A},A/\fq)) \Rightarrow \Tor^A_{p+q}(A/\fp,A/\fq) \]
therefore degenerates to give isomorphisms $\Tor_p^{\overline{A}}(A/\fp,\overline{A/\fq}) \cong \Tor_p^A(A/\fp, A/\fq)$ for all $p \geq 0$. We obtain the following chain of inequalities:
\[ e(A/\fp)e(\overline{A/\fq}) \leq \chi^{\overline{A}}(A/\fp,\overline{A/\fq}) = \chi^A(A/\fp, A/\fq) = e(A/\fp)e(A/\fq) \leq e(A/\fp)e(\overline{A/\fq}). \]
Noting that $\overline{A}$ is equicharacteristic, we apply Theorem \ref{fulton_blowup} to conclude that \newline $\gr(A/\fp) \otimes_{\gr \overline{A}} \gr(\overline{A/\fq}) = \gr(A/\fp) \otimes_{\gr A} \gr(\overline{A/\fq})$ is $0$-dimensional. Since we also have $e(A/\fq) = e(\overline{A/\fq})$, Proposition \ref{redproj} says that 
\[ \Proj \left( \frac{\gr A}{\opi \gr A} \otimes_{\gr A} \gr(A/\fq) \right)_{\red} = \Proj \left( \gr \left( \overline{A/\fq} \right) \right)_{\red}, \]
meaning that
\[ \Proj \left( \gr(A/\fp)  \otimes_{\gr A} \frac{\gr A}{\opi \gr A} \otimes_{\gr A} \gr(A/\fq) \right)_{\red} = \Proj \left( \gr (A/\fp) \otimes_{\gr A} \gr \left( \overline{A/\fq} \right) \right)_{\red} \]
is empty. Because $\opi$ annihilates $\gr(A/\fp)$, we see that 
\[\gr(A/\fp) \otimes_{\gr A} \gr(A/\fq) = \gr(A/\fp)  \otimes_{\gr A} \frac{\gr A}{\opi \gr A} \otimes_{\gr A} \gr(A/\fq)\]
is $0$-dimensional (Lemma \ref{st_explained}).
\end{proof}

The case of two $R_0$-flat modules is considerably more complicated. Before attempting to prove Conjecture I, we prove an intermediate result which is interesting in its own right:

\begin{mthm}With $A$ as above, let $M$ and $N$ be $R_0$-flat $A$-modules such that $\dim M + \dim N = \dim A$ and $\ell(M \otimes_A N) < \infty$. If $\chi^A(M,N)= e(M)e(N)$, then $e(M) = e(M/\pi M)$ or $e(N) = e(N/\pi N)$.
\end{mthm}
\begin{proof}Once again, we may assume that $A = \hat{A} = R[[X_1, \cdots, X_n]]$ where $R$ is a complete DVR with uniformizer $\pi$. From the proof of Theorem \ref{unramified_ineq}, $\chi^A(M,N)$ is just the multiplicity of $M \oth_R N$ with respect to the diagonal ideal $\fd$. That is,
\[ \chi^A(M,N) = e_\fd(M \oth_R N) \geq e(M \oth_R N) \geq e(M)e(N). \]
If $\chi^A(M,N) = e(M)e(N)$, then in particular, $e(M \oth_R N) = e(M)e(N)$. By Theorem \ref{ctp_eq}, the result now follows.
\end{proof}

\begin{mthm}With $A$ as above, let $M = A/\fa$ and $N = A/\fb$ be equidimensional, $R_0$-flat quotients of $A$ such that $\dim M + \dim N = \dim A$ and $\ell(M \otimes_A N) < \infty$. Suppose that $\chi^A(M,N) = e(M)e(N)$ and $e(M) = e(M/\pi M)$. Then $\dim(\gr M \otimes_{\gr A} \gr N) = 0$ (i.e. Conjecture I holds) in the presence of any one of the following conditions:
	\begin{itemize}
		\item[(i)] $\opi$ is not contained in a height-one associated prime of $\gr M$;
		\item[(ii)] $e(N) = e(N / \pi N)$;
		\item[(iii)] $\dim M = \dim A - 1$;
	  \item[(iv)] $\dim M = 1$;
		\item[(v)] $\ds \dim(\gr M \otimes_{\gr A} \gr N \otimes_{\gr A} \frac{\gr A}{\opi \gr A}) = 0$.
	\end{itemize}
\end{mthm}

\begin{proof} Arguing as in the proofs of Theorems B and C, we are free to assume that $A = \hat{A} = R[[X_1, \cdots, X_n]]$ where $R$ is a complete DVR with uniformizer $\pi$. Since $e(A/\fa) = e((A/\fa)/\pi(A/\fa))$, Theorem \ref{ctp_eq} gives us
\[e(A/\fa \oth_R A/\fb) = e(A/\fa)e(A/\fb) = \chi^A(A/\fa,A/\fb) = e_\fd(A/\fa \oth_R A/\fb) \]
where $\fd$ is the ideal generated by $\left\{X_1 \oth 1 - 1 \oth X_i \right\}_{1 \leq i \leq n}$. As $A/\fa \oth_R A/\fb$ is equidimensional (Corollary \ref{equidim_cor}), Rees's Theorem (\ref{rees_thm}) tells us that the integral closure of $\fd$ in $A/\fa \oth_R A/\fb$ is the entire maximal ideal. As in the proof of Theorem \ref{fulton_blowup}, this is enough to conclude that $\ds \frac{\gr(A/\fa \oth_R A/\fb)}{(\cdots, \overline{X_i \otimes 1 - 1 \otimes X_i}, \cdots)}$ is $0$-dimensional. In cases (i), (ii), and (iii), Proposition \ref{gr_R} assures us that the canonical surjection
\[ \frac{\gr (A/\fa) \otimes_k \gr(A/\fb)}{\dopi} \twoheadrightarrow \gr(A/\fa \oth_R A/\fb) \]
induces an homeomorphism on $\Spec$, so $\gr(A/\fa) \otimes_{\gr A} \gr(A/\fb)$, the quotient of $\ds \frac{\gr (A/\fa) \otimes_k \gr(A/\fb)}{\dopi}$ by the $\overline{X_i} \otimes 1 - 1 \otimes \overline{X_i}$, will be $0$-dimensional as claimed.

In case (v), let $\fp$ and $\fq$ be arbitrary minimal primes of $A/\fa$ and $A/\fb$. Since $A/\fa$ and $A/\fb$ are equidimensional, we necessarily have $\chi^A(A/\fp,A/\fq) = e(A/\fp)e(A/\fq)$. By hypothesis, $\ds \dim(\gr(A/\fp) \otimes_{\gr A} \gr(A/\fq) \otimes_{\gr A} \frac{\gr A}{\opi \gr A}) = 0$, so by Proposition \ref{good_locus}, $\dim(\gr(A/\fp) \otimes_{\gr A} \gr(A/\fq)) = 0$. Since $\fp$ and $\fq$ were arbitrary, the conclusion follows.

In case (iv), recall that Proposition \ref{ss_mult} says that going modulo $\opi$ will drop the dimension of $\gr(A/\fa)$. Since $\dim \gr(A/\fa) = 1$, we are reduced to case (v) where the result is already known.
\end{proof}

\bibliographystyle{alpha}
\bibliography{unramified_paper}

\end{document}